\documentclass[11pt]{article}
\usepackage[koi8-r]{inputenc}
\usepackage{amssymb,amsmath,amsfonts,cite}
\usepackage[russian]{babel}
\usepackage[koi8-r]{inputenc}
\input{mfpic.tex}
\usepackage{amsthm}
\usepackage{graphicx}
\usepackage{float}
\allowdisplaybreaks[2]
\numberwithin{equation}{section}
\newtheorem{theorem}{Теорема}[section]
\newtheorem{lemma}[theorem]{Лемма}

\newtheorem{remark}[theorem]{Замечание}

\providecommand{\abs}[1]{\lvert #1\rvert}

\newcommand{\nc}{\newcommand}
\nc{\vb}{\mathbf{v}}
\nc{\bx}{\mathbf{x}}
\nc{\by}{\mathbf{y}}
\nc{\bz}{\mathbf{z}}
\nc{\bu}{\mathbf{u}}
\nc{\ba}{\mathbf{a}}
\nc{\bs}{\mathbf{s}}
\nc{\bq}{\mathbf{q}}
\nc{\bd}{\mathbf{d}}
\nc{\bb}{\mathbf{b}}
\nc{\bfr}{\mathbf{r}}
\nc{\bA}{\mathbf{A}}
\nc{\R}{\mathbb R}
\nc{\N}{\mathbb N}
\nc{\bbC}{\mathbb C}
\nc{\bbD}{\mathbb D}
\nc{\F}{\mathbf F}
\nc{\bbS}{\mathbb S}
\nc{\B}{\cal B}
\nc{\br}{\bigr}
\nc{\bl}{\bigl}
\nc{\Bl}{\Bigl}
\nc{\Br}{\Bigr}
\nc{\ind}{\mathbf{1}}
  
\title{
\begin{flushright}
\mbox{\small{Посвящается С.В. и Е.М.}}
\end{flushright}
Геометрия больших очередей}
\author{А.А. Пухальский}

\begin{document}
\maketitle
\sloppy
\begin{abstract}
Предлагается подход к 
 нахождению  наиболее вероятных траекторий,
ведущих к образованию длинных очередей в 
эргодической сети Джексона, основанный на
решении уравнений Гамильтона. Так как соответствующий гамильтониан является
разрывным и кусочно липшицевым,
возникает  необходимость использовать аппарат негладкого анализа.
Обращение уравнений Гамильтона во времени приводит к уравнениям
жидкостной динамики двойственной сети. Соответственно,
 оптимальные траектории есть
обращённые во времени
 жидкостные траектории двойственной сети.  Эти траектории с необходимостью 
проходят через области,
удовлетворяющие некоторому условию ''существенности''. 
Результаты проиллюстрированы на примере сети Джексона из
двух узлов. Установлены также свойства субстохастических матриц,
которые могут представлять самостоятельный интерес.
\end{abstract}
\section{Оптимальные траектории
  в сети Джексона и 
вариационные задачи с нефиксированным временем}
\label{sec:introduction}
Принцип больших уклонений позволяет находить грубые асимптотики
вероятностей достижения случайным процессом нетипичных значений и
наиболее вероятные сценарии осуществления  редких событий.  
Для этого необходимо решить вариационную задачу минимизации функции
уклонений (называемую также ''функционалом действия'') по траекториям,
которые реализуют такие   события.
 Как правило, ищется решение уравнения Эйлера--Лагранжа, см., напр.,
 Вентцель, Фрейдлин \cite{wf-r}, Shwartz, Weiss
\cite{SchWei95}. 
 В данной работе на примере  сети Джексона изучаются
 возможности подхода,
 состоящего в решении уравнений Гамильтона.

Рассматривается (экспоненциальная) сеть Джексона из  $K$ узлов с матрицей
маршрутизации
 $P=(p_{kl})$ спектрального радиуса меньше единицы, с интенсивностями входящих
 потоков   $\lambda_k>0$ и интенсивностями обслуживания в узлах
   $\mu_k>0$\,,  см., напр., Клейнрок \cite{Kle79-r}.
Обозначим через $Q_k(t)$ длину очереди в узле $k$ в
момент $t$ и пусть $Q(t)=(Q_1(t),\ldots,Q_K(t))^T$\,,
где $^T$ -- оператор транспонирования. 
Для $J\subset\{1,2,\ldots,K\}$ обозначим через $F_J$ множество
векторов $x=(x_1,\ldots,x_K)^T$ таких, что $x_i=0$ при $i\in J$ и $x_i>0$
при $i\notin J$\,. Как обычно принято, 
$J^c=\{1,2,\ldots,K\}\setminus J$\,.
Обозначим также
 $  \pi(u)=u\ln
 u-u+1$\,,  если $u>0\,$, и положим $\pi(0)=1$\,.
  Пусть  $\ind_\Gamma$ представляет собой
 индикаторную функцию множества $\Gamma$\,, которая равна  $1$ на
 $\Gamma$ и равна $0$ вне $\Gamma$.
Обозначим через   $\bbS^{K\times K}_+$ множество субстохастических (по строкам)
 матриц размера $K\times K$, через $I$ -- единичную матрицу размера
 $K\times K$\,, а через $\mathbb D(\R_+,
\R^K_+)$ -- пространство Скорохода непрерывных справа и имеющих пределы
слева функций, заданных на $\R_+$ и принимающих значения в $\R_+^K$\,,
которое снабжено метрикой, превращающей его в полное метрическое
сепарабельное пространство, см., Жакод, Ширяев \cite{jacshir-r},
Ethier, Kurtz \cite{EthKur86}.
Следующий  результат получен
 в Puhalskii \cite{Puh07}. (Предполагается, что произведение функции
 $\pi$ с числом 0 равно нулю даже если аргумент функции $\pi$
равен бесконечности и что $0/0=0$\,.)
\begin{theorem}
  \label{co:jackson}
Пусть  $Q(0)=0\in\R_+^K$\,.
Тогда, при  $n\to\infty$, процессы
 $(Q(nt)/n,\,t\in\R_+)$ удовлетворяют принципу больших уклонений 
в $\bbD(\R_+,\R^K_+)$ с
 фyнкцией уклонений $\mathbf{I}(q)$, 
которая для абсолютно непрерывных функций  $q=(q(t),\,t\in\R_+)$\,,
таких что
 $q(0)=0\in\R_+^K$\,, имеет вид
\begin{equation*}
  \mathbf{I}(q)=\int_0^\infty
L(q(t),\dot q(t))\,dt,
\end{equation*}
где
\begin{align}
  L(x,y)&=\sum_{J\subset\{1,2,\ldots,K\}}
\ind_{F_J}(x)L_J(y)\,,\notag\\
 \label{eq:35a}  L_J(y)&
=\inf_{\substack{(a,d,\varrho)\in\R^K_{+}
\times\R_+^K\times\mathbb{S}_+^{K\times K}
:\\y=a+(\varrho^T-I)d}}\psi_J(a,d,\varrho)
\intertext{и}
   \psi_J(a,d,\varrho)&=\sum_{k=1}^K
\pi\bl(\frac{a_k}{\lambda_k}\br)\,\lambda_k
+\sum_{k\in J^c}
\pi\bl(\frac{d_k}{\mu_k}\br)\,\mu_k
+\sum_{k\in J} \pi\bl(\frac{d_k}{\mu_k}\br)\,
\ind_{(\mu_k,\infty)}(d_k)\mu_k+\notag\\ \label{eq:35b}
&+\sum_{k=1}^K d_k \biggl[\sum_{l=1}^K \pi\Bl(\frac{\varrho_{kl}}{p_{kl}}\Br)p_{kl}+
\pi\Bl( \frac{1-\sum_{l=1}^K
  \varrho_{kl}}{1-\sum_{l=1}^K p_{kl} }\Br)\bl(1-\sum_{l=1}^K
p_{kl}\br)\biggr].
\end{align}
Если функция  $q$ не является абсолютно непрерывной или 
если $q(0)\not=0$, то
$\mathbf{I}(q)=\infty$. 
\end{theorem}
 В развёрнутом виде утверждение теоремы означает, что 
множества $\{q\in \bbD(\R_+,\R^K):\,\mathbf I(q)\le u\}$
являются компактами при всех $u\ge 0$ и выполняются неравенства 
\begin{equation*}
  \liminf_{n\to\infty}\frac{1}{n}\,\ln\mathbf P(q\in G)\ge -\inf_{q\in
    G}\mathbf I(q)
\end{equation*}
при условии, что $G$-- открытое подмножество $\bbD(\R_+,\R_+^K)$\,, и
\begin{equation*}
  \limsup_{n\to\infty}\frac{1}{n}\,\ln\mathbf P(q\in F)\le -\inf_{q\in
    F}\mathbf I(q)
\end{equation*}
при условии, что $F$-- замкнутое подмножество $\bbD(\R_+,\R^K_+)$\,.
Заметим, что формулируемый в теореме 1.1 принцип больших уклонений
установлен также  в
Atar, Dupuis\cite{MR1719274}, 
Ignatiouk--Robert \cite{Ign00}, где
функция $L_J$ даётся в другом виде.

Из принципа больших уклонений следует, что, в широких предположениях,
 логарифм вероятности попадания
процесса $Q(nt)/n$ в заданное множество $S$ на больших интервалах времени
 близок к $-\inf_{q\in S'}\mathbf I(q)$\,, где
$S'=\{q:\,q(0)=0\,, q(T)\in S \text{ для некоторого }T>0\}$\,. Если
последний инфимум достигается  единственной функцией $ q $\,, то
попадание $Q(nt)/n$
 во множество $S$ происходит с вероятностью близкой к единице
 движением в малой окрестности этой функции.
В данной статье основное внимание  уделяется вопросу нахождения для
эргодической сети Джексона
 такой функции $q$, начинающейся в начале координат и 
достигающей фиксированной точки в некоторый момент $T$\,, для которой 
 $\int_0^TL(q(t),\dot q(t))\,dt$ минимален. Показано, что обращение во
 времени уравнений Гамильтона для оптимальных траекторий
 приводит к уравнениям жидкостной динамики двойственной сети.
Это наблюдение позволяет подтвердить для
эргодической сети Джексона  справедливость
''фольклорной теоремы'', утверждающей, что
 наиболее вероятная
 траектория, обеспечивающая ''большое уклонение''
 эргодического марковского
 процесса, является обращённой во времени 
 траекторией жидкостного предела двойственного процесса, см.,
 напр., Collingwood
\cite{Col15}.
(В статфизической литературе соответствующее свойство известно как
обобщённый принцип Онзагера--Махлупа: траектория флуктуации есть
обращённая во времени траектория релаксации двойственной динамики, см.,
напр., Bouchet, Laurie, Zaboronsky \cite{BouLauZab14}.)
   Показано, что найденная оптимальная
траектория представляет собой предел (условного) 
закона больших чисел при условии, что
длины очередей в сети достигают аномально больших значений.
Найдено необходимое и достаточное условие, при выполнении которого
данная грань ортанта $\R_+^K$ может содержать часть
 оптимальной траектории, а, следовательно, -- жидкостной
 траектории. В заключение, на примере
  эргодической сети Джексона из двух узлов рассматриваются
  геометрические аспекты построения оптимальных траекторий.
Кроме того, в статье установлены свойства субстохастических матриц, которые
могут представлять самостоятельный интерес.

Так как функции $L_J(y)$ являются выпуклыми по лемме 4.2 в Puhalskii \cite{Puh07}, то, ввиду неравенства
Йенсена, оптимальное движение в области с постоянной динамикой, т.е.,
в грани $F_J$\,, происходит по прямой.
Заметим также, что если $J'\subset J''$\,,  то $\psi_{J'}(a,d,\varrho)\ge
\psi_{J''}(a,d,\varrho)$\,и, как следствие, 
$L_{J'}(y)\ge
L_{J''}(y)$\,. Поэтому, если начальная и конечная точки 
участка траектории
принадлежат $F_J$\,, то движение по прямой ''выгоднее'', чем движение
через смежные области
$F_{J'}$\,, где $J'\subset J$\,. Будет показано, что, более того,
 оптимальные
траектории  состоят из конечного числа
 отрезков, которые принадлежат 
 граням $F_J$ с убывающими по включению множествами $J$\,,
если не считать конечную точку, для которой свойство убывания $J$ может не
выполняться.  В
 частности, каждая из граней $F_J$ содержит не более одного отрезка.

Для нахождения оптимальных траекторий в гранях $F_J$ будет
использоваться
  следующий общий результат.
Для функции $f(y)$\,,
удовлетворяющей условию Липшица локально, будем обозначать через $\partial f(y)$
обобщённый градиент $f$ в точке $y$\,, см., Кларк \cite{Cla83-r}. Если
$f$ выпукла, что верно для приложений в данной статье,
 то обобщённый градиент совпадает с субдифференциалом, см., напр.,
 Рокафеллар \cite{Rock-r}. Определение 
 конуса нормального к непустому замкнутому множеству в точке этого множества
как сопряжённого к касательному конусу в этой точке можно найти на с.19 в Кларк
\cite{Cla83-r}. Для случая выпуклого множества определение сводится к
определению на с.32
в Рокафеллар \cite{Rock-r} как совокупности векторов нормальных к
множеству в рассматриваемой точке, где вектор $z$
называется нормальным к выпуклому множеству $S$ в точке $x\in S$\,,
если $z\cdot (y-x)\le0$ для всех $y\in S$\,. (Здесь и ниже
$\cdot$ используется для обозначения скалярного произведения.
Кроме того, сокращение ''п.в.'' означает ''почти всюду по мере Лебега''.)
\begin{lemma}
  \label{le:lagrange}
Пусть $L(x,y)$ --  функция
из $\R^k\times \R^k$ в $\R$\,, удовлетворяющая условию
Липшица по $(x,y)$ локально и выпуклая по $y$\,.
Предположим, что  функция
\begin{equation}
  \label{eq:31}
    H(x,p)=\sup_{y\in\R^k}(p\cdot y-L(x,y))
\end{equation}
удовлетворяет  условию Липшица локально (по $(x,p)$)\,.
Пусть $g(x)$ -- полунепрерывная сверху  функция,
удовлетворяющая условию Липшица локально. 
Пусть $S$ -- выпуклое замкнутое подмножество
$\R^k$ и $x_0\in\R^k$\,.
Если минимум 
$\int_0^TL(x(t),\dot x(t))\,dt$ по $T\ge0$ и по абсолютно непрерывным
 функциям $x(t)$ таким, что $x(0)=x_0$\,, $x(T)\in S$ и
$g(x(t))\le0$ для всех $t\le T$\,,   достигается для $T= T^\ast $
и  $x(t)=x^\ast(t)$\,, 
то  существуют  мера $\mu$ на $[0, T^\ast ]$\,, $\mu$--измеримая
 функция
$\gamma(t)$  
и   абсолютно непрерывная функция $p(t)$\,, обе принимающие значения в
$\R^k$\,, такие
что имеют место следующие свойства:
\begin{enumerate}
\item $\gamma(t)\in \partial^>g(x^\ast(t))$ для $\mu$--почти всех
  $t\in[0, T^\ast ]$ и носитель меры $\mu$ содержится в множестве
  $\{t\in[0, T^\ast ]:\,\partial^>g( x^\ast(t))\not=\emptyset\}$\,,
где $\partial^>g(x)$ представляет собой выпуклую оболочку пределов
последовательностей $\gamma_i$\,, удовлетворяющих условиям 
$\gamma_i\in \partial g(x_i)$\,, $x_i\to x$\,, $g(x_i)>0$\,,
  \item \begin{equation*}
\left[  \begin{array}[c]{c}
    -\dot p(t)\\\dot  x^\ast(t)
  \end{array}\right]\in \partial H\bl(x^\ast(t),p(t)+\int_0^t
\gamma(s)\, d\mu(s)\br)
\quad \text{п.в. на } [0,T^\ast],
\end{equation*}
\item 
  \begin{equation*}
    H(x^\ast(t),p(t)+\int_0^t\gamma(s)\, d\mu(s))=0\quad\text{  на } [0,T^\ast]\,,
  \end{equation*}
\item
  \begin{equation*}
    p( T^\ast )+\int_0^{ T^\ast } \gamma(s)\,d\mu(s)\in -N_S(x^\ast( T^\ast ))\,,
  \end{equation*}
где $N_S(x^\ast( T^\ast ))$ -- нормальный конус к $S$ в точке $x^\ast( T^\ast )$\,. \end{enumerate}
\end{lemma}
Доказательство леммы отнесено в приложение.
Её применение облегчается тем, что, в отличие от лагранжиана $L_J$\,,
соответствующий гамильтониан 
$H_J(\theta)=\sup_{y\in\R^K}(\theta\cdot y-L_J(y))$\,, где $\theta\in\R^K$\,, имеет сравнительно
простой вид.
Пусть для $\theta=(\theta_1,\ldots,\theta_K)^T$ и $k=1,2,\ldots,K$
\begin{equation}
  \label{eq:27}
  h_k(\theta)=e^{-\theta_k}\bl(\sum_{l=1}^K(e^{\theta_l}-1)p_{kl}+1\br)-1\,.
\end{equation}
Выкладки, приводимые в приложении, показывают, что
\begin{equation}
  \label{eq:35}
  L_J(y)=
\sup_{\theta\in\R^K}
\Bl(\sum_{k=1}^K
\theta_k\,y_k
-\sum_{k=1}^K(e^{\theta_k}-1)\lambda_k
-\sum_{k\in J}h_k(\theta)^+\mu_k-\sum_{k\in J^c}h_k(\theta)\mu_k\,
\Br)\,,
\end{equation}
где используется обозначение $u^+=\max(u,0)$\,.
Отсюда, в частности, вытекает, что функция $L_J$ выпукла и
полунепрерывна снизу. Кроме того,
\begin{equation}
    H_J(\theta)=
  \label{eq:69}
    \sum_{k=1}^K(e^{\theta_k}-1)\lambda_k+\sum_{k\in J}h_k(\theta)^+\mu_k+\sum_{k\in J^c}h_k(\theta)\mu_k\,
\,.
\end{equation}
Обозначим также
\begin{equation}
  \label{eq:28}
   H_0
   (\theta)=H_\emptyset(\theta)=\sum_{k=1}^K(e^{\theta_k}-1)\lambda_k
+\sum_{k=1}^Kh_k(\theta)\mu_k\,.
\end{equation}
Следующее утверждение 
вытекает из леммы  \ref{le:lagrange} с
$L(x,y)=L_J(y)$ и с $g(x)=\sum_{i\in J} x_i^2$\,,
так что мера $\mu$ равна нулю.
\begin{lemma}
  \label{le:urnie}
Oптимальное движение в грани $F_J$ происходит с постоянным 
импульсом $\theta$ таким, что $H_J(\theta)=0$\,, а
уравнение для оптимальных траекторий имеет вид:
\begin{equation}
  \label{eq:24}
  \dot q(t)\in\partial H_J(\theta) \text{ п.в.}
\end{equation}

\end{lemma}
Эта лемма оставляет открытым вопрос о существовании оптимальных
траекторий. К сожалению, в отличии от случая фиксированного интервала
времени в рассматриваемой
нами постановке существование оптимальной траектории не следует из
 компактности множеств $\{q:\,\mathbf I(q)\le
u\}$\,, где $u\ge0$\,.  Ниже будет предполагаться, по большей части,
 что оптимальная траектория
существует и будут исследоваться её свойства.
 Существование и единственность будут доказаны  
 с помощью {\em ad hoc} рассуждений.
\section{Свойства гамильтониана}
\label{sec:-hamilt}

В этом разделе найдены кандидаты для оптимальных импульсов
 $\theta$ при движении в
гранях $F_J$\,. Введём предварительно некоторые обозначения.
Для квадратной матрицы $B$
 через $B(i|j)$ обозначается матрица,
получающаяся из $B$ выбрасыванием $i$--й строки и $j$--го столбца.
Аналогично, обозначим через 
$b_{\cdot
      l}$ $l$--й столбец матрицы $B$ без $l$--го элемента, а через
    $b_{l\cdot}$ -- $l$--ую строку без $l$--го элемента.

Пусть вектор
$\theta^{(m)}=(\theta^{(m)}_1,\ldots,\theta^{(m)}_K)^T\not=0$\,, где 
$m\in\{1,2,\ldots,K\}$\,, удовлетворяет
уравнениям $h_k(\theta^{(m)})=0$ при $k\not=m$ и $H_0 (\theta^{(m)})=0$\,.
Ввиду \eqref{eq:27}, отношения $a_{ml}=(e^{\theta^{(m)}_l}-1)/(e^{\theta^{(m)}_m}-1)$
при $m\not=l$ удовлетворяют системе уравнений 
\begin{equation}
  \label{eq:3}
a_{mk}-  \sum_{l\not=m}a_{ml}p_{kl}=p_{km}\,, k\not=m\,.
\end{equation}
(Заметим, что
если $\theta^{(m)}_m=0$\,, то $\theta^{(m)}_k=0$ для всех $k$ в силу
\eqref{eq:27}.)
Уравнения \eqref{eq:3} имеют единственное решение
\begin{equation}
  \label{eq:8}
a_{m\cdot}^T=\bl((I-P)(m|m)\bl)^{-1}p_{\cdot m}\,.
\end{equation}
Кроме того, $a_{mm}=1$\,.
Числа $a_{ml}$ определяются этими условиями единственным образом и
являются неотрицательными. 
 Подстановка в условие $H_0
(\theta^{(m)})=0$\,, см. \eqref{eq:28},
показывает, что 
\begin{equation}
  \label{eq:4}
  e^{\theta^{(m)}_m}=\dfrac{1-\displaystyle\sum_{l=1}^Ka_{ml}p_{ml}}{\displaystyle\sum_{l=1}^Ka_{ml}\lambda_l}\,\mu_m\,.
\end{equation}
Обозначим
$\theta^\ast =(\theta^{(1)}_1,\theta^{(2)}_2,\ldots,\theta^{(K)}_K)^T$
и
\begin{equation}
  \label{eq:25}
  \nu=(I-P^T)^{-1}\lambda\,,
\end{equation}
где $\lambda=(\lambda_1,\ldots,\lambda_K)^T$\,.
В стационарном режиме  $\nu=(\nu_1,\ldots,\nu_K)^T$ -- это вектор 
средних потоков через узлы сети. 
Положим $\rho_m=\nu_m/\mu_m$\,, где $m=1,2,\ldots,K$\,.
Следующая теорема показывает, в частности, что точка пересечения 
гиперплоскостей 
$\{\theta:\,\theta_m=\theta_m^{(m)}\}$ лежит на поверхности $H_0(\theta)=0$\,.
\begin{theorem}
\label{the:H} Имеют место соотношения  $H_0 (\theta^\ast )=0$ и 
$e^{-\theta^{(m)}_m}=\rho_m$ для $m=1,2,\ldots,K$\,.  
\end{theorem}
Предпошлем доказательству ряд лемм. Будем предполагать, что
$B=(b_{ij})$ --  $K\times K$--матрицa  с ненулевым определителем и 
ненулевыми главными минорами.
Ниже $\text{adj}$ используется для обозначения присоединенной матрицы, 
$\text{det}$ -- для обозначения определителя,
а через
$M_{ij}(l|l)$ обозначается минор $(i,j)$--го элемента матрицы $B(l|l)$\,.
Обозначим для $l\not=m$
\begin{equation}
  \label{eq:10}
  f_{ml}=
  \begin{cases}
    e_m\,,&\text{ если }m<l\,,\\
e_{m-1}\,,&\text{ если }m>l\,,
  \end{cases}
\end{equation}
где 
$e_i$ представляет собой $i$--ый вектор стандартного базиса в $\R^{K-1}$\,.
\begin{lemma}
  \label{le:adj}
Если $l\not=m$\,, то
\begin{equation*}
  f_{ml}^T\,\text{adj}(B(l|l))b_{\cdot l}=(-1)^{m+l+1}\,\text{det}\,(B(l|m))\,.
\end{equation*}
\end{lemma}
\begin{proof}
  Пусть $l>m$\,. Имеем, что 
\begin{equation*}
  e_m^T\,\text{adj}\,(B(l|l))b_{\cdot l}
=\sum_{j=1}^{l-1}(-1)^{m+j}M_{jm}(l|l)b_{jl}+
\sum_{j=l+1}^{K}(-1)^{m+j-1}M_{j-1,m}(l|l)b_{jl}\,.
\end{equation*}
Так как $M_{jm}(l|l)=M_{j,l-1}(l|m)$ при $j\le l-1$ и
$M_{j-1,m}(l|l)=M_{j-1,l-1}(l|m)$ при $j\ge l+1$\,, получаем, что
$e_m^T\,\text{adj}\,(B(l|l))b_{\cdot l}=(-1)^{m+l+1}\text{det}\,(B(l|m))\,.$

Рассмотрим случай $l<m$\,. Рассуждая аналогичным образом, имеем, что
\begin{multline*}
e_{m-1}^T\,\text{adj}\,(B(l|l))b_{\cdot l}=
\sum_{j=1}^{l-1}(-1)^{m+j-1}M_{j,m-1}(l|l)b_{jl}+
\sum_{j=l+1}^K (-1)^{m+j}M_{j-1,m-1}(l|l)b_{jl}=\\
=\sum_{j=1}^{l-1}(-1)^{m+j-1}M_{jl}(l|m)b_{jl}+
\sum_{j=l+1}^K (-1)^{m+j}M_{j-1,l}(l|m)b_{jl}=(-1)^{m+l+1}\,\text{det}\,(B(l|m))\,.
\end{multline*}
\end{proof}
\begin{lemma}
  \label{le:det}
Для произвольного $l=1,2,\ldots,K$  \begin{equation*}
    b_{ll}\,\text{det}(B(l|l))-b_{l\cdot}\text{adj}(B(l|l))b_{\cdot
      l}=\text{det}\,(B)\,.  
  \end{equation*}
\end{lemma}
\begin{proof}
  Так как 
  \begin{equation*}
    b_{l\cdot}\,\text{adj}(B(l|l))b_{\cdot l}=
\sum_{j\not=l}b_{lj}f_{jl}^T\,\text{adj}\,(B(l|l))b_{\cdot l}\,,
  \end{equation*}
то, применяя лемму \ref{le:adj}, получаем, что
\begin{equation*}
      b_{l\cdot}\,\text{adj}(B(l|l))b_{\cdot l}
=\sum_{j\not=l}b_{lj}(-1)^{j+l+1}\,\text{det}\,(B(l|j))=
-\text{det}\,(B)+b_{ll}\,\text{det}\,(B(l|l))\,.
\end{equation*} \end{proof}
\begin{lemma}
  \label{le:tozd}
Имеют место равенства
\begin{equation}
  \label{eq:13}
    \frac{b_{m\cdot}\bl(B(m|m)\br)^{-1}b_{\cdot m}}{b_{mm}
-b_{m\cdot}\bl(B(m|m)\br)^{-1}b_{\cdot m}}
=\sum_{l\not=m}\frac{b_{lm}f_{ml}^T\bl(B(l|l)\br)^{-1}b_{\cdot l}}{b_{ll}
-b_{l\cdot}\bl(B(l|l)\br)^{-1}b_{\cdot l}}
\end{equation}
и
\begin{equation}
  \label{eq:20}
  -\,  \frac{b_{mm}f_{lm}^T(B(m|m))^{-1}b_{\cdot m}}{b_{mm}
-b_{m\cdot}(B(m|m))^{-1}b_{\cdot m}}
=-\,\frac{b_{lm}}{b_{ll}-b_{l\cdot}(B(l|l))^{-1}b_{\cdot l}}\\
+\sum_{\substack{k\not=l,\\k\not=m}}\frac{b_{km}f_{lk}^T(B(k|k))^{-1}b_{\cdot
    k}}
{b_{kk}-b_{k\cdot}(B(k|k))^{-1}b_{\cdot k}}\,.
\end{equation}
\end{lemma}
\begin{proof}
Умножая числители и знаменатели в обеих частях \eqref{eq:13} на
определители обращаемых матриц, видим, что \eqref{eq:13} равносильно
равенству
\begin{equation}
  \label{eq:17}
      \frac{b_{m\cdot}\,\text{adj}\,(B(m|m))b_{\cdot
          m}}{b_{mm}\,\text{det}\,(B(m|m)) 
-b_{m\cdot}\,\text{adj}\,(B(m|m))b_{\cdot m}}
=\sum_{l\not=m}\frac{b_{lm}f_{ml}^T\,\text{adj}\,(B(l|l))
b_{\cdot l}}{b_{ll}\,\text{det}\,(B(l|l))
-b_{l\cdot}\,\text{adj}\,(B(l|l))b_{\cdot l}}\,.
\end{equation}
По лемме \ref{le:det}, знаменатели в \eqref{eq:17} равны 
$\text{det}(B)\not=0$\,.
Поэтому нужно доказать, что 
\begin{equation*}
  b_{m\cdot}\,\text{adj}\,(B(m|m))b_{\cdot
          m}
=\sum_{l\not=m}b_{lm}f_{ml}^T\,\text{adj}\,(B(l|l))
b_{\cdot l}\,.
\end{equation*}
Применяя лемму \ref{le:adj}, имеем, что
\begin{equation*}
  \sum_{l\not=m}b_{lm}f_{ml}^T\,\text{adj}\,(B(l|l))
b_{\cdot l}=\sum_{l\not=m}b_{lm}(-1)^{m+l+1}\,\text{det}(B(l|m))=
-\text{det}(B)+b_{mm}\,\text{det}(B(m|m))\,,
\end{equation*}
что завершает доказательство \eqref{eq:13} ввиду леммы \ref{le:det}.

Аналогично доказательству \eqref{eq:13} имеем, что \eqref{eq:20} эквивалентно равенству
\begin{equation}
  \label{eq:21}
  -b_{mm}f_{lm}^T\,\text{adj}\,(B(m|m))b_{\cdot m}
=-b_{lm}\,\text{det}\,(B(l|l))
+\sum_{\substack{k\not=l,\\k\not=m}}b_{km}f_{lk}^T\,\text{adj}\,(B(k|k))b_{\cdot
    k}\,.
\end{equation}
По лемме \ref{le:adj}
\begin{equation}
  \label{eq:39}
    \sum_{k\not=l}b_{km}f_{lk}^T\,\text{adj}\,(B(k|k))b_{\cdot
    k}=\sum_{k\not=l}b_{km}
(-1)^{k+l+1}\,\text{det}\,(B(k|l))\,.
\end{equation}
Так как $\sum_{k=1}^Kb_{km}
(-1)^{k+l}\,\text{det}\,(B(k|l))=0$ ввиду того, что
сумма в левой части
является определителем матрицы, получающейся из матрицы $B$ заменой
$l$--го столбца на $m$--й, то правая часть \eqref{eq:39} равна
$b_{lm}\textit{det}(B(l|l))\,.$ Полученное равенство эквивалентно \eqref{eq:21}.

\end{proof}

\begin{proof}[Доказательство теоремы \ref{the:H}]
Будем доказывать сначала второе утверждение.
Покажем, что \begin{align}
  \label{eq:7}
  \frac{1}{1-\sum_{l=1}^Ka_{ml}p_{ml}}&=1+
\sum_{k=1}^K\frac{a_{km}p_{km}}{1-\sum_{l=1}^K
  a_{kl}p_{kl}}\intertext{ и }
  \frac{a_{ml}}{1-\sum_{k=1}^Ka_{mk}p_{mk}}
&=\sum_{k=1}^K\frac{a_{kl}p_{km}}{1-\sum_{l'=1}^K
  a_{kl'}p_{kl'}}\text{ при условии, что } l\not=m\,.  \label{eq:7a}
\end{align}
Так как 
 $a_{mm}=1$\,, то равенство
 \eqref{eq:7} переписывается в виде
\begin{equation}
  \label{eq:9}
    \frac{\sum_{l\not=m}a_{ml}p_{ml}}{1-p_{mm}-\sum_{l\not=m}a_{ml}p_{ml}}=
\sum_{k\not=m}\frac{a_{km}p_{km}}{1-p_{kk}-\sum_{l\not=k}
  a_{kl}p_{kl}}\,.
\end{equation}
В силу \eqref{eq:3},  \eqref{eq:8} и
 \eqref{eq:10}, 
$a_{ml}=f_{lm}^T((I-P)(m|m))^{-1}p_{\cdot m}$ и
$a_{km}=f_{mk}^T((I-P)(k|k))^{-1}p_{\cdot
  k}$\,. Это позволяет  записать 
\eqref{eq:9} в виде
\begin{equation*}
  \frac{p_{m\cdot}\bl((I-P)(m|m)\br)^{-1}p_{\cdot m}}{1-p_{mm}
-p_{m\cdot}\bl((I-P)(m|m)\br)^{-1}p_{\cdot m}}
=\sum_{k\not=m}\frac{p_{km}f_{mk}^T\bl((I-P)(l|l)\br)^{-1}p_{\cdot k}}{1-p_{kk}
-p_{k\cdot}\bl((I-P)(l|l)\br)^{-1}p_{\cdot k}}\,,
\end{equation*}
что является частным случаем \eqref{eq:13} в утверждении
 леммы \ref{le:tozd} при $B=I-P$\,.
Равенство \eqref{eq:7} доказано.

Докажем \eqref{eq:7a}.
Перепишем это равенство в виде
\begin{equation}
  \label{eq:18}
  \frac{a_{ml}(1-p_{mm})}{1-p_{mm}-\sum_{k\not=m}a_{mk}p_{mk}}
=\frac{p_{lm}}{1-p_{ll}-\sum_{l'\not=l}
  a_{ll'}p_{ll'}}
+\sum_{\substack{k\not=l,\\k\not=m}}\frac{a_{kl}p_{km}}{1-p_{kk}-\sum_{l'\not=k}
  a_{kl'}p_{kl'}}\,,\text{ где } l\not=m\,.
\end{equation}
Как и выше, в силу  \eqref{eq:8} и \eqref{eq:10},
 \eqref{eq:18} эквивалентно соотношению
\begin{multline*}
    \frac{(1-p_{mm})f_{lm}^T((I-P)(m|m))^{-1}p_{\cdot m}}{1-p_{mm}
-p_{m\cdot}((I-P)(m|m))^{-1}p_{\cdot m}}
=\frac{p_{lm}}{1-p_{ll}-p_{l\cdot}((I-P)(l|l))^{-1}p_{\cdot l}}+\\
+\sum_{\substack{k\not=l,\\k\not=m}}\frac{p_{km}f_{lk}^T((I-P)(k|k))^{-1}p_{\cdot
    k}}
{1-p_{kk}-p_{k\cdot}((I-P)(k|k))^{-1}p_{\cdot k}}
\end{multline*}
и выполнено ввиду \eqref{eq:20} в утверждении леммы \ref{le:tozd}. Это завершает доказательство \eqref{eq:7a}.

Положим
\begin{equation}
  \label{eq:30}
  c_{mk}=  \dfrac{a_{mk}}{1-\sum_{l=1}^Ka_{ml}p_{ml}}
\end{equation}
и $C=(c_{mk})$\,.
Покажем, что
\begin{equation}
  \label{eq:26}
  C=(I-P^T)^{-1}\,.
\end{equation}
В силу \eqref{eq:30}, соотношение \eqref{eq:7a} принимает вид
$c_{ml}=\sum_{k=1}^K p_{km}c_{kl}$\,, так что $(m,l)$--элементы матриц $C$ и
$P^TC$ совпадают, если $m\not=l$\,. Так как $a_{mm}=1$\,, сoотношение
\eqref{eq:30} позволяет переписать \eqref{eq:7} в виде
$c_{mm}=1+\sum_{k=1}^Kp_{km}c_{km}$\,. Следовательно, диагональные элементы
матриц $C$ и $I+P^TC$ совпадают.
Таким образом, $C=I+P^TC$\,, что доказывает \eqref{eq:26}.
Так как $a_{mm}=1$\,, то, в силу \eqref{eq:30},
$1-\sum_{l=1}^Ka_{ml}p_{ml}=1/c_{mm}$ и 
\begin{equation}
  \label{eq:37}
  a_{mk}=\frac{c_{mk}}{c_{mm}}\,.
\end{equation}
Ввиду \eqref{eq:4}, \eqref{eq:25} и \eqref{eq:26},
$e^{\theta^{(m)}_m}=\mu_m/\sum_{k=1}^Kc_{mk}\lambda_k=
\mu_m/\nu_m=1/\rho_m$\,.

Для завершения доказательства теоремы \ref{the:H} заметим, что, в силу
\eqref{eq:27}, \eqref{eq:28} и определения $\theta^\ast$\,,
  \begin{equation*}
H_0 (\theta^\ast )= \sum_{m=1}^K(e^{\theta^{(m)}_m}-1)\bl(
\lambda_m+\sum_{k=1}^Ke^{-\theta^{(k)}_k}
p_{km}\mu_k-e^{-\theta^{(m)}_m}\mu_m\br)\,.
  \end{equation*}
Так как $e^{-\theta^{(k)}_k}=\nu_k/\mu_k$ и $\nu=\lambda+P^T\nu$\,, то
$\lambda_m+\sum_{k=1}^Ke^{-\theta^{(k)}_k}
p_{km}\mu_k-e^{-\theta^{(m)}_m}\mu_m=0$\,.
\end{proof}
Пусть
$\tilde\theta^{(J)}=(\tilde\theta^{(J)}_1,\ldots,\tilde\theta^{(J)}_K)^T$
задаётся соотношениями
 $\tilde\theta^{(J)}_k=\theta^{(k)}_k$\,, если 
$k\not\in J$\,, и $h_k(\tilde\theta^{(J)})=0$ при $k\in J$\,.
Заметим, что если $J=\{1,2,\ldots,K\}\setminus \{m\}$\,, то
$\tilde\theta^{(J)}=\theta^{(m)}$\,.
Ввиду \eqref{eq:69} и \eqref{eq:28},
\begin{equation}
  \label{eq:70}
  H_J(\tilde\theta^{(J)})=H_0 (\tilde\theta^{(J)})\,.
\end{equation}
Если движение происходит в гранях
$F_{J_1},\ldots,F_{J_k}$ со смещениями $s_1,\ldots,s_k$
в течение промежутков времени $t_1,\ldots,t_k$\,, соответственно,
 то,   так как $ H_0 (\tilde\theta^{(J)})=0$\,, как будет показано в
следующей лемме, так как $\tilde\theta^{(J)}\cdot x=
\theta^\ast \cdot x$ при $x\in F_J$ и так как имеет место \eqref{eq:70}\,, 
  затраты равны 
  \begin{multline}
    \label{eq:11}
   \sum_{l=1}^kt_lL_{J_l}\bl(\frac{s_l}{t_l}\br)= \sum_{l=1}^k
  \sup_{\theta\in\R^K}(\theta\cdot s_l-t_lH_{J_l}(\theta))
\ge \sum_{l=1}^k\bl(
\tilde\theta^{(J_l)}\cdot s_l-t_lH_{J_l}(\tilde\theta^{(J_l)})\br)=\\
=\sum_{l=1}^k\br(\theta^\ast \cdot s_l-t_lH_0(\tilde\theta^{(J_l)})\br)
=\theta^\ast \cdot (s_1+\ldots +s_k)=\theta^\ast\cdot r\,,
\end{multline}
где через $r$ обозначена точка, в которую необходимо попасть.
В общем случае, если $q(0)=0$ и $q(t)=r$\,, то, с учётом того, что
$\dot q_k(s)=0$ для $k\in J$ п.в. на множестве $\{s:\,q(s)\in F_J\}$\,,
\begin{multline}
  \label{eq:41}
  \int_0^tL(q(s),\dot q(s))\,ds=
\sum_J\int_0^t\mathbf1_{F_J}(q(s))L_J(\dot q(s))\,ds\ge
\sum_J\int_0^t\mathbf1_{F_J}(q(s))
\bl(\tilde\theta^{(J)} \cdot \dot
q(s)-H_0(\tilde\theta^{(J)})\br)\,ds=
\\=\sum_J\int_0^t\mathbf1_{F_J}(q(s))
\bl(\theta^\ast \cdot \dot
q(s)-H_0(\tilde\theta^{(J)})\br)\,ds=\theta^\ast\cdot r\,.
\end{multline}
Если траектория из $O$ в $r$ такова,
что в \eqref{eq:11} и \eqref{eq:41} достигаются равенства,
т.е., выполнены равенства~\eqref{eq:24} с $\theta=\tilde\theta^{(J)}$\,, то эта траектория
оптимальна. Поскольку функции $H_{J}(\theta)$ строго выпуклы, так
что их
 производные инъективны, то такая траектория будет единственной
оптимальной. 

Следующая объявленная лемма обобщает соотношение
$H_0(\theta^\ast)=0$\,. 
\begin{lemma}
\label{le:subspace}
Имеет место равенство $ H_0 (\tilde\theta^{(J)})=0$\,.  
\end{lemma}
\begin{proof}
Ввиду \eqref{eq:27}, \eqref{eq:28}, определения $\tilde\theta^{(J)}$ и
равенства $e^{-\theta^{(m)}_m}=\rho_m$ \,,
  \begin{multline}
    \label{eq:48}
     H_0 (\tilde\theta^{(J)})=
\sum_{m\not\in J}(e^{\theta^{(m)}_m}-1)\lambda_m+
\sum_{m\in J}(e^{\tilde\theta^{(J)}_m}-1)\lambda_m+\\+
\sum_{k\not\in J}\Bl(e^{-\theta^{(k)}_k}\bl(
\sum_{m\not\in J}(e^{\theta^{(m)}_m}-1)p_{km}+
\sum_{m\in J}(e^{\tilde\theta^{(J)}_m}-1)p_{km}+1\br)-1\Br)\mu_k=\\
=\sum_{m\not\in J}(e^{\theta^{(m)}_m}-1)\bl(\lambda_m+\sum_{k\not\in
  J}e^{-\theta^{(k)}_k}p_{km}\mu_k -e^{-\theta^{(m)}_m}\mu_m\br)+
\sum_{m\in J}(e^{\tilde\theta^{(J)}_m}-1)\bl(\lambda_m
+\sum_{k\not\in J}e^{-\theta^{(k)}_k}p_{km}\mu_k\br)=\\
=\sum_{m\not\in J}(e^{\theta^{(m)}_m}-1)\bl(\lambda_m+\sum_{k\not\in
  J}p_{km}\nu_k -\nu_m\br)+
\sum_{m\in J}(e^{\tilde\theta^{(J)}_m}-1)\bl(\lambda_m
+\sum_{k\not\in J}p_{km}\nu_k\br)\,.
  \end{multline}
Так как ввиду определения $\nu$ (см. \eqref{eq:25}), 
$    \lambda_m+\sum_{k\not\in J}p_{km}\nu_k=\nu_m-\sum_{k\in J}p_{km}\nu_k\,,
$ то 
\begin{equation}
  \label{eq:49}
    \sum_{m\in J}(e^{\tilde\theta^{(J)}_m}-1)\bl(\lambda_m
+\sum_{k\not\in J}p_{km}\nu_k\br)=
  \sum_{m\in J}(e^{\tilde\theta^{(J)}_m}-1)\nu_m-
 \sum_{k\in J}\nu_k \sum_{m\in J}(e^{\tilde\theta^{(J)}_m}-1)p_{km}\,.
\end{equation}
Так как $h_k(\tilde\theta^{(J)})=0$ при $k\in J$\,, то 
\begin{equation}
  \label{eq:50}
    \sum_{m\in J}(e^{\tilde\theta^{(J)}_m}-1)p_{km}=
e^{\tilde\theta^{(J)}_k}-1-\sum_{m\not\in J}(e^{\theta^{(m)}_m}-1)p_{km}\,.
\end{equation}
Таким образом, правая часть \eqref{eq:49} равна
$  \sum_{m\not\in J} (e^{\theta^{(m)}_m}-1)\sum_{k\in J}p_{km}\nu_k\,.$
В силу \eqref{eq:48},
\begin{equation*}
   H_0 (\tilde\theta^{(J)})=
\sum_{m\not\in J}(e^{\theta^{(m)}_m}-1)\bl(\lambda_m+\sum_{k=1}^Kp_{km}\nu_k -\nu_m\br)\,.
 \end{equation*}
Так как 
$  \lambda_m+\sum_{k=1}^Kp_{km}\nu_k -\nu_m=0
$ ввиду \eqref{eq:25}, то $H_0 (\tilde\theta^{(J)})=0$\,.
\end{proof}
\section{Существенные и несущественные грани}
\label{sec:sush}
В этом разделе исследуется движение 
 в гранях и приводится
основной результат. 
Скажем, что грань $F_J$ является
существенной, если траектория  с импульсом
$\tilde\theta^{(J)}$\,, которая проходит через $F_J$\,, существует. В
противном случае грань называется несущественной.
Поскольку $\dot q_k(t)=0$ п.в. при $k\in J$\,, в силу леммы
\ref{le:urnie},
 если грань $F_J$ существенна, то
\begin{equation}
  \label{eq:32}
  0\in\partial_k H_J(\tilde
\theta^{(J)})\text{ при }k\in J\,.
\end{equation}
Заметим, что, см. \eqref{eq:27},
\begin{equation}
  \label{eq:12}
      \partial_kh_k(\theta)=
-e^{-\theta_k}(1-\sum_{l=1}^Kp_{kl}+\sum_{l\not=k}e^{\theta_l}p_{kl})
\end{equation}
и
\begin{equation}
  \label{eq:19}
    \partial_l h_k(\theta)=e^{-\theta_k}e^{\theta_l}p_{kl}\,, \;l\not=k\,.
\end{equation}
Как следствие,
\begin{equation}
  \label{eq:34}
      \partial_lh_k(\theta)>0\,, \text{ если } l\not=k
 \text{ и } \partial_kh_k(\theta)<0\,.
\end{equation}
Заметим также, что если $h_k(\theta)=0$\,, то
 $\partial h_k(\theta)^+$ -- это множество векторов 
$\alpha \nabla h_k(\theta)$\,, где $\alpha\in[0,1]$\,.

  Ввиду \eqref{eq:24} с $\theta=\tilde \theta^{(J)}$\,, \eqref{eq:12},
   \eqref{eq:19} и равенства $h_l(\tilde\theta^{(J)})=0$ при $l\in J$\,, если
  грань $F_J$ является существенной, то 
 существуют числа $\alpha^{(J)}_l\,,l\in J\,,$ из интервала
$[0,1]$ такие, что, при $k\in J$\,, п.в.,
\begin{multline*}
0=\dot q_k(t)=    e^{\tilde\theta^{(J)}_k}\lambda_k+
\sum_{l\not\in J}\partial_k h_l(\tilde\theta^{(J)})\mu_l
+\sum_{l\in J}\alpha_l^{(J)}\partial_k h_l(\tilde\theta^{(J)})\mu_l=
e^{\tilde\theta^{(J)}_k}\lambda_k+
\sum_{l\not\in J}e^{-\theta^{(l)}_l}e^{\tilde\theta^{(J)}_k}p_{lk}\mu_l+\\
+\sum_{\substack{l\in
    J,\\l\not=k}}\alpha_l^{(J)}e^{-\tilde\theta^{(J)}_l}e^{\tilde\theta^{(J)}_k}p_{lk}\mu_l
-\alpha^{(J)}_ke^{-\tilde\theta^{(J)}_k}(1-\sum_{l=1}^Kp_{kl}+\sum_{l\not=k}e^{\tilde\theta^{(J)}_l}p_{kl})\mu_k\,.
\end{multline*}
Так как $h_k(\tilde \theta^{(J)})=0$\,, то
\begin{equation}
  \label{eq:15}
      1-\sum_{l=1}^Kp_{kl}+\sum_{l\not=k}e^{\tilde\theta^{(J)}_l}p_{kl}
=e^{\tilde\theta^{(J)}_k}(1-p_{kk})\,.
\end{equation}
 Поэтому требуется, чтобы при $k\in J$\,,
\begin{equation*}
    \lambda_k+
\sum_{l\not\in J}p_{lk}\nu_l
+\sum_{l\in
    J}\alpha_l^{(J)}e^{-\tilde\theta^{(J)}_l}p_{lk}\mu_l
-\alpha_k^{(J)}e^{-\tilde\theta^{(J)}_k}\mu_k
=0\,.
\end{equation*}
Так как 
$\lambda_k+
\sum_{l\not\in J}p_{lk}\nu_l=\nu_k-\sum_{l\in J}p_{lk}\nu_l$\,, то,
эквивалентным образом,
\begin{equation*}
\alpha_k^{(J)}e^{-\tilde\theta^{(J)}_k}\mu_k-  \sum_{l\in
    J}\alpha_l^{(J)}e^{-\tilde\theta^{(J)}_l}p_{lk}\mu_l
=\nu_k-\sum_{l\in J}p_{lk}\nu_l\,.
\end{equation*}
В векторной форме,
\begin{equation*}
  (I-P^T)(J^c|J^c)(\alpha_l^{(J)}e^{-\tilde\theta^{(J)}_l}\mu_l\,, l\in J)^T=
(I-P^T)(J^c|J^c)(\nu_l\,,l\in J)^T\,.
\end{equation*}
Так как матрица $(I-P^T)(J^c|J^c)$ невырождена, то при $l\in J$
\begin{equation}
  \label{eq:54}
  \alpha_l^{(J)}e^{-\tilde\theta^{(J)}_l}\mu_l=\nu_l\,.
\end{equation}
 Поэтому для того, 
чтобы грань $F_J$ была существенной, необходимо, чтобы
\begin{equation}
  \label{eq:55}
\tilde\rho^{(J)}_l\ge \rho_l\,,l\in J,
\end{equation}
где обозначено
\begin{equation}
  \label{eq:43}
    \tilde\rho^{(J)}_l=e^{-\tilde\theta^{(J)}_l}\,. 
\end{equation}

Так как ввиду  \eqref{eq:50},
\begin{equation*}
  (I-P)(J^c|J^c)((\tilde\rho^{(J)}_l)^{-1}-1\,, l\in
  J)^T=P(J^c|J)(\rho_l^{-1}-1\,, l\not\in J)^T\,,
\end{equation*}
то
\begin{equation}
  \label{eq:45}
    ((\tilde\rho^{(J)}_l)^{-1}-1,\,l\in J)^T=\bl((I-P)(J^c|J^c)\br)^{-1}
P(J^c|J)(\rho_l^{-1}-1,\,l\not\in J)^T\,.
\end{equation}
Поскольку матрица $\bl((I-P)(J^c|J^c)\br)^{-1}
P(J^c|J)$ неотрицательна, то $\tilde\rho^{(J)}_l\le1$ при $l\in J$\,.
Кроме того, условие существенности  \eqref{eq:55} эквивалентно тому,
что, покомпонентно,
\begin{equation}
  \label{eq:46}
      (\rho_l^{-1}-1\,, l\in J)^T\ge
\bl((I-P)(J^c|J^c)\br)^{-1}
P(J^c|J)(\rho_l^{-1}-1,\,l\not\in J)^T\,.
\end{equation}
Если 
 $k\notin J$\,, то, ввиду \eqref{eq:24} с
 $\theta=\tilde\theta^{(J)}$\,, \eqref{eq:25}, \eqref{eq:12}, \eqref{eq:19},
и \eqref{eq:54}, для траектории с импульсом $\tilde\theta^{(J)}$ в
 грани $F_J$ имеем, что
 \begin{multline}
   \label{eq:61}
   \dot{  q}_k(t)=  e^{\theta^{(k)}_k}\lambda_k+
\sum_{\substack{l\not\in J,\\l\not=k}}e^{-\theta^{(l)}_l}e^{\theta^{(k)}_k}p_{lk}\mu_l
+\sum_{l\in
    J}\alpha_l^{(J)}e^{-\tilde\theta^{(J)}_l}e^{\theta^{(k)}_k}p_{lk}\mu_l-\\
-e^{-\theta^{(k)}_k}(1-\sum_{l=1}^Kp_{kl}+\sum_{l\not=k}e^{\tilde\theta^{(J)}_l}p_{kl})
\mu_k
= \rho_k^{-1}\bl(\lambda_k+
\sum_{l=1}^Kp_{lk}\nu_l\br)
-(1+\sum_{l=1}^K(e^{\tilde\theta^{(J)}_l}-1)p_{kl})
\nu_k=\\
= \mu_k
-(1+\sum_{l\notin J}(\rho_l^{-1}-1)p_{kl}+
\sum_{l\in J}((\tilde\rho^{(J)}_l)^{-1}-1)p_{kl})
\nu_k\,.
\end{multline}
Аналогичные рассуждения (или формальная подстановка $J=\emptyset$ в
последнее выражение в \eqref{eq:61}) пoказывают, что для любого $k$
\begin{equation}
  \label{eq:62}
    \partial_k H_0 (\theta^\ast )
=\mu_k-
\bl(1+\sum_{l=1}^K(\rho_l^{-1}-1)p_{kl}\br)\nu_k\,,
\end{equation}
так что, при $k\not\in J$\,,
\begin{equation}
  \label{eq:64}
  \dot{  q}_k(t)=\partial_k H_0 (\theta^\ast )+
\sum_{l\in J}(\rho_l^{-1}-(\tilde\rho^{(J)}_l)^{-1}) p_{kl}
\nu_k\,.
\end{equation}
Так как, ввиду \eqref{eq:15} и \eqref{eq:43}, для $k\in J$ \eqref{eq:62}   может быть записано в виде
\begin{equation}
  \label{eq:1}
    \partial_k H_0 (\theta^\ast )=(\rho_k^{-1}-(\tilde\rho_k^{(J)})^{-1})\nu_k
-\sum_{l=1}^Kp_{kl}(\rho_l^{-1}-(\tilde\rho_l^{(J)})^{-1})\nu_k
\end{equation}
и $\dot q_k(t)=0$ п.в.,
то п.в. для $k\in J$
\begin{equation}
  \label{eq:29}
    \dot{  q}_k(t)=\partial_k H_0 (\theta^\ast )-
(\rho_k^{-1}-(\tilde\rho_k^{(J)})^{-1})\nu_k+
\sum_{l\in J}(\rho_l^{-1}-(\tilde\rho^{(J)}_l)^{-1}) p_{kl}
\nu_k\,.
\end{equation}
 Напомним определение двойственной сети
Джексона.
В такой сети интенсивности
обслуживания в узлах те же, что в исходной сети. Потоки
двойственной сети имеют те же интенсивности, что и в
исходной, но направлены в противоположную сторону. Таким образом,
обозначая параметры двойственной сети теми же символами, но с верхней
чертой, имеем, что $\overline \mu_k=\mu_k$\,, $\overline \nu_k=\nu_k$
и $\overline\nu_k \overline p_{kl}=\nu_l p_{lk}$\,. Также меняются
местами входящие и выходящие потоки:
$\overline \lambda_k=\nu_k(1-\sum_{l=1}^K p_{kl})$ и 
$ \lambda_k=\nu_k(1-\sum_{l=1}^K\overline p_{kl})$\,. 
(Следующее рассуждение показывает, что матрица $\overline P=(\overline p_{kl})$
является субстохастической. Так как $\nu=\lambda+P^T\nu$\,, то $\nu\ge
P^T\nu$ покомпонентно. Поэтому $\nu_k\ge \sum_{l=1}^K p_{lk}\nu_l$\,, т.е.,
$\sum_{l=1}^K\overline p_{kl}=\sum_{l=1}^K p_{lk}\nu_l/\nu_k\le 1$\,.)
Если сеть Джексона находится в стационарном состоянии, то
обращение во времени приводит к стационарной двойственной
сети.

В этих обозначениях,
в силу \eqref{eq:64} и \eqref{eq:29}, если $ q (t)\in F_J$\,, то п.в.
\begin{equation*}
  \dot{  q}(t)=\nabla H_0 (\theta^\ast )-
(I-\overline P^T)\varphi^{(J)}\,,
\end{equation*}
где
\begin{equation}
  \label{eq:66}
  \varphi^{(J)}_l=
  \begin{cases}
    (\rho_l^{-1}-(\tilde\rho^{(J)}_l)^{-1})\nu_l\,,
&\text{ если } l\in J\,,\\
0\,,&\text{ в противном случае.}
  \end{cases}
\end{equation}
Кроме того, \eqref{eq:62} принимает вид
\begin{equation}
  \label{eq:14}
    -\nabla H_0 (\theta^\ast )=\overline\lambda-(I-\overline{P}^T)\mu\,,
\end{equation}
где $\mu=(\mu_1,\ldots,\mu_K)^T$\,.
(Заметим, что правую часть также можно записать как $\nabla \overline
H_0(0)$\,.) Пусть $T>0$\,.
 Обозначая  $\overline q(t)= q (T-t)$\,, где $0\le t\le T$\,,
  получаем, что п.в.
\begin{equation}
  \label{eq:67}
  \dot{\overline{q}}(t)=\overline\lambda -(I-\overline P^T)\mu+
(I-\overline P^T)\dot\phi(t)\,,
\end{equation}
где
\begin{equation}
  \label{eq:68}
  \dot\phi(t)=\sum_{J}\ind_{F_J}(\overline q(t))\varphi^{(J)}\,.
\end{equation}
Так как $\dot\phi_k(t)=0$\,, если $\overline q_k(t)>0$ ввиду \eqref{eq:66}, то \eqref{eq:67}
означает, что  $\overline q(t)$ получается косым отражением функции 
$q(T)+(\overline\lambda -(I-\overline P^T)\mu)t$\,, т.е., является
жидкостной траекторией, начинающейся в конечной точке $r$\,, длин очередей в
 сети Джексона с интенсивностями входных потоков
$\overline \lambda_k$\,, интенсивностями обслуживания $\mu_k$ и
матрицей переходов $\overline P$\,. Покажем обратное: если $\overline
q$ -- жидкостная траектория двойственной системы, начинающаяся в
конечной точке $r$\,,
то её обращение во времени есть оптимальная траектория. Так как $\overline
q$ -- это решение задачи косого отражения, то имеет местo
\eqref{eq:67}\,, где $\phi(t)$ -- покомпонентно неубывающая абсолютно
непрерывная функция, такая что $ \overline
q_k(t)\dot\phi_k(t)=0$ п.в. Ввиду \eqref{eq:14}, если $\overline q(t)\in F_J$\,, то
для сужений векторов на $J$ и сужений матриц на $J\times J$ будем
иметь, что п.в.
$\dot\phi_J(t)=(I_{J,J}-P^T_{J,J})^{-1}\bl(-\nabla H_0(\theta^\ast)_J\br)$\,.
В силу \eqref{eq:1} и \eqref{eq:66},
$\dot\phi_J(t)=\varphi^{(J)}_J$\,, т.е.,  выполнено \eqref{eq:68}.
Таким образом, обращение траектории 
$\overline q$ проходит через $F_J$ с импульсом $\tilde \theta^{(J)}$\,. Поэтому грани
$F_J$ в \eqref{eq:68} -- существенные. Мы также имеем, что
 $\varphi^{(J)}_l\ge0$ при $l\in J$ и что \eqref{eq:55} не только
 необходимо, но и достаточно для существенности $F_J$\,, поскольку
 жидкостная траектория, начинающаяся в $F_J$\,, будет там находиться
 некоторое время, если \eqref{eq:55} выполнено.
Следовательно, обращение оптимальной траектории с импульсами
$\tilde\theta^{(J)}$\,, которая
попадает из некоторой точки $A$ в некоторую точку $B$\,, является
жидкостной траекторией для попадания из $B$ в $A$ для двойственной
сети и наоборот.  

Так как исходная сеть Джексона предполагается эргодической, то
двойственная  сеть  тоже эргодична, так что
 жидкостная траектория
двойственной сети, начинающаяся в точке $r$\,, попадает, в конце
концов, в начало координат. Более того, соответствующая траектория
состоит из конечного числа отрезков, принадлежащих граням $F_J$ с
возрастающими
 по
включению множествами $J$\,, кроме, возможно ''одноточечного отрезка'',
представляющего собой начальную точку, 
см., напр., лемму 5.3 на с.142 и лемму
5.4 на с.143 в Bramson \cite{Bra06}.
Поэтому оптимальная траектория из $O$ в $r$ существует, 
 единственна и состоит из
конечного числа отрезков, принадлежащих граням $F_J$ с
 убывающими по включению
множествами $J$\,, кроме, возможно ''одноточечного отрезка'',
представляющего собой конечную точку, ср. рис. 8 ниже.
  Пусть $T^\ast$ -- время, за которое жидкостная траектория
 $\overline
q(t)$\,, начинающаяся в конечной точке $r$\,, попадает в
начало координат. Положим $ q^\ast (t)=\overline q(T^\ast-t)$ при $0\le
t\le  T^\ast $\,.  (Напомним, что $ q^\ast $
достигает нижней границы затрат, ср. \eqref{eq:11}.)
Нами  доказано следующее утверждение. 
\begin{theorem}
  \label{the:optimum}
Построенная траектория $ q^\ast (t)$ является единственной  оптимальной с точностью до времени, проведённого  в начале координат.
Движение в существенной грани $F_J$ происходит с импульсом $\tilde
\theta^{(J)}$\,.    Затраты на
достижение  точки  $r$ при
движении по этой траектории равны
$\int_0^{ T^\ast } L( q^\ast (t),\dot{  q^\ast }(t))\,dt=\theta^\ast \cdot
r$\,. 
\end{theorem}

Как следствие теоремы \ref{the:optimum}, получаем, что жидкостная
траектория $\overline q$ 
двойственной сети, начинающаяся в  точке 
$r$\,, достигает начала координат, проходя через последовательность
существенных граней, в каждой из которых она проводит положительное
время.  Если начальное значение  принадлежит
несущественной грани, то  траектория сразу же уходит из этой
грани. В существенной  грани $F_J$ движение происходит в соответствии
с уравнениями $\dot{\overline q}_k(t)=-\partial_k H_0 (\theta^\ast )-
\sum_{l\in J}(\rho_l^{-1}-(\tilde\rho^{(J)}_l)^{-1}) p_{kl}
\nu_k$\,, если $k\not\in J$\,, и $\dot{\overline q}_k(t)=0$ п.в., если
$k\in J$\,, ср., \eqref{eq:61} и \eqref{eq:64}.
Заметим также, что поскольку 
$  \dot{\overline q}(t)=\overline \lambda+
(\overline P^T-I)\overline{d}^{(J)}\,,
$ где через $\overline d^{(J)}
=(\overline d_1^{(J)},\ldots,\overline d^{(J)}_K)^T$ обозначен
вектор интенсивностей потоков,
протекающих через узлы, при
$\overline q(t)\in F_J$\,, то, в силу \eqref{eq:67} и \eqref{eq:68},
 $\overline d^{(J)}=\mu-\varphi^{(J)}$\,. Ввиду
\eqref{eq:66}, $\overline d^{(J)}_l=\mu_l$\,, если $l\notin рJ$\,, и 
$\overline d_l^{(J)}=(\tilde{\rho}^{(J)}_l)^{-1}\nu_l$\,, если $l\in J$\,.
Так как $1\ge\tilde{\rho}^{(J)}_l\ge\rho_l$ при условии, что грань 
$F_J$ существенна, то в этом случае $\nu_l\le\overline d^{(J)}_l\le
\mu_l$\,.
В силу \eqref{eq:35a} и \eqref{eq:35b}, отсюда следует, что 
для двойственной сети $\overline L_J(\overline\lambda+
(\overline P^T-I)\overline d^{(J)})=0$\,. Как следствие, 
\begin{equation}
  \label{eq:33}
   \overline
L(\overline q(t),\dot{\overline q}(t))=0 \text{ п.в.}
\end{equation}

Условие существенности имеет относительно простой вид для случая полуоси.
Если грань $F_J$ -- это полуось $x_m$\,, то \eqref{eq:45} принимает вид
\begin{equation}
  \label{eq:23}
    (\tilde\rho^{(J)}_l)^{-1}-1=a_{ml}(\rho_m^{-1}-1)\,,  l\not=m,
\end{equation}
а \eqref{eq:46}, с учётом \eqref{eq:8}, --
\begin{equation}
  \label{eq:42}\rho_l^{-1}-1\ge
a_{ml}(\rho_m^{-1}-1)\,,  l\not=m\,.
\end{equation}
(Разумеется, \eqref{eq:42}
 можно получить напрямую из \eqref{eq:54}.)
Подстановка \eqref{eq:23} в \eqref{eq:61} с учётом того, что, как
показано в доказательстве теоремы \ref{the:H},
$\sum_{l=1}^ma_{ml}p_{ml}=1-1/c_{mm}$\,, показывает, что
если полуось $x_m$ -- существенна, то
\begin{equation*}
\dot{  q}^\ast_m(t)=
\frac{1-\rho_m}{c_{mm}}\,\mu_m>0\,,
\end{equation*}
как и следовало ожидать.
Следующее свойство матрицы $C$ доказано в приложении.
\begin{lemma}
  \label{le:sush_os'}
Имеет место неравенство
$  c_{ml}\le c_{mm}\,.$
\end{lemma}
Из  леммы \ref{le:sush_os'}, \eqref{eq:37} и  \eqref{eq:42} следует, что
 если $\rho_m$ максимально, то
  полуось $x_m$ является существенной, т.е., в случае попадания на неё
  оптимальная траектория будет находиться там некоторое время.

\begin{remark}
 Из  теоремы \ref{the:H} следует, что $\theta^\ast>0$ покомпонентно и,
как следствие, нижняя граница в \eqref{eq:11} и \eqref{eq:41} положительна для всех $r$
 тогда и только тогда, когда рассматриваемая сеть эргодична.
 Таким образом,  для неэргодической сети Джексона
 на оптимальность импульсов
 $\tilde \theta^{(J)}$ при движении из $O$ в $r$ рассчитывать не
 приходится.
С другой стороны, сделанное наблюдение, что обращение уравнений
Гамильтона приводит к жидкостным уравнениям, остаётся в силе и оценки
\eqref{eq:11} и \eqref{eq:41} могут быть полезны. Например, если
$\theta^\ast<0$ покомпонентно, то жидкостные траектории двойственной
сети убегают на бесконечность, а затраты на переход  исходной сети
  из точки $r$ в
начало координат равны $-\theta^\ast\cdot r$\,.
\end{remark}
\section{Условные законы больших чисел.  Обсуждение.}
\label{sec:ld}

Как отмечено вначале, принцип больших уклонений позволяет получать
логарифмические асимптотики вероятностей редких событий и находить
наиболее вероятные сценарии, реализующие эти события.  
Следующие две леммы иллюстрируют это утверждение.
Напомним, что сеть предполагается эргодической.
\begin{lemma}
  \label{le:traject}
Для произвольного $\delta>0$\,,
\begin{equation*}
  \lim_{\epsilon\to0}\limsup_{n\to\infty}
\mathbf P(\sup_{t\le  T^\ast }\abs{\frac{Q(nt)}{n}- q^\ast (t)}\ge \delta|
\,\abs{\frac{Q(n T^\ast )}{n}-r}\le\epsilon)=0\,.
\end{equation*}
Кроме того,
\begin{equation*}
  \lim_{\epsilon\to0}\limsup_{n\to\infty}\frac{1}{n}\,\ln
\mathbf P(\abs{\frac{Q(n T^\ast )}{n}-r}\le\epsilon)=
\lim_{\epsilon\to0}\liminf_{n\to\infty}\frac{1}{n}\,\ln
\mathbf P(\abs{\frac{Q(n T^\ast )}{n}-r}\le\epsilon)=-\theta^\ast \cdot r\,.
\end{equation*}
\end{lemma}
\begin{proof}
Для доказательства первого утверждения  достаточно доказать, что
  \begin{equation}
    \label{eq:36}
    \limsup_{\epsilon\to0}\limsup_{n\to\infty}\frac{1}{n}\ln
\mathbf P(\sup_{t\le  T^\ast }\abs{\frac{Q(nt)}{n}- q^\ast (t)}\ge \delta|
\abs{\frac{Q(n T^\ast )}{n}-r}\le\epsilon)<0\,.
  \end{equation}
В силу принципа больших уклонений,
\begin{equation}
  \label{eq:40}
    \liminf_{n\to\infty}\frac{1}{n}\ln
\mathbf P(\abs{\frac{Q(n T^\ast )}{n}-r}\le\epsilon)\ge
-\inf_{q:\,\abs{q( T^\ast )-r}\le\epsilon/2}\mathbf
I(q)\,.
\end{equation}
Продолжим $ q^\ast $ на интервал времени $( T^\ast ,\infty)$
по закону больших чисел. Аналогично \eqref{eq:33}, 
$\int_{ T^\ast }^\infty L( q^\ast (t),\dot{ q}^\ast(t))\,dt=0$\,. Поэтому
 $\mathbf I( q^\ast )=\int_0^{ T^\ast }L( q^\ast (t),\dot{ q}^\ast(t))\,dt$\,.
Так как функция $\mathbf I(q)$ компактна снизу и $ q^\ast ( T^\ast )=r$\,, то 
\begin{equation}
  \label{eq:38}
  \lim_{\epsilon\to0}\inf_{q:\,\abs{q( T^\ast )-r}\le\epsilon/2}\mathbf
I(q)=\inf_{q:\,q( T^\ast )=r}\mathbf
I(q)\le \mathbf I( q^\ast )=\int_0^{ T^\ast }L( q^\ast (t),\dot{ q}^\ast(t))\,dt\,.
\end{equation}
Аналогично,
\begin{equation*}
\limsup_{\epsilon\to0}
\limsup_{n\to\infty}\frac{1}{n}\ln
  \mathbf P(\sup_{t\le  T^\ast }\abs{\frac{Q(nt)}{n}- q^\ast (t)}\ge \delta\,,\;
\abs{\frac{Q(n T^\ast )}{n}-r}\le\epsilon)\le
-\inf_{\substack{q:\,\sup_{t\le  T^\ast }\abs{q(t)- q^\ast (t)}\ge \delta\,,\\
q( T^\ast )=r}}\mathbf I(q)\,.
\end{equation*}
Пусть последний $\inf$ достигается в точке $\tilde
q$\,. Так как $\inf_{q,\,T:\,q(T)=r}\int_0^{T}L(q(t),\dot
q(t))\,dt$ 
достигается на
единственной траектории $ q^\ast $
 и $\tilde q\not= q^\ast $\,, то 
$\mathbf I(\tilde q)\ge \int_0^{ T^\ast }L(\tilde q(t),\dot
{\tilde q}(t))\,dt>\int_0^{ T^\ast }L( q^\ast (t),\dot{ q}^\ast(t))\,dt$\,.
Таким образом,
\begin{equation*}
  \limsup_{\epsilon\to0}
\limsup_{n\to\infty}\frac{1}{n}\ln
  \mathbf P(\sup_{t\le  T^\ast }\abs{\frac{Q(nt)}{n}- q^\ast (t)}\ge \delta\,,\;
\abs{\frac{Q(n T^\ast )}{n}-r}\le\epsilon)
<-\int_0^{ T^\ast }L( q^\ast (t),\dot{ q}^\ast(t))\,dt\,.
\end{equation*}
Вспоминая \eqref{eq:40} и \eqref{eq:38}, получаем \eqref{eq:36}.
Второе утверждение леммы следует, по существу, из \eqref{eq:40}.
\end{proof}
В качестве ещё одного применения рассмотрим момент достижения
суммарной очередью большого значения:
$\tau_n=\inf\{t:\,\sum_{k=1}^KQ_k(nt)\ge nA\}$\,.
Определим $r$ как точку минимума $\theta^\ast \cdot x$ по всем
$x=(x_1,\ldots,x_K)^T\in\R_+^K$\,, 
таким что $\sum_{k=1}^K x_k=A$\,. ''В общем положении'' такая точка
$r$ определяется единственным образом.
\begin{lemma}
  \label{le:long_time}
 Для всех
достаточно больших $T$
\begin{equation}
  \label{eq:22}
 \lim_{n\to\infty}\frac{1}{n}\,
\ln\mathbf P( \tau_n\le T)=
 -\theta^\ast\cdot r
\end{equation}
и
\begin{equation}
  \label{eq:16}
  \lim_{n\to\infty}\mathbf P(\sup_{(\tau_n-T^\ast)^+\le t\le \tau_n}\abs{\frac{Q(nt)}{n}-
    q^\ast(t-(\tau_n-T^\ast)^+)}<\delta| \tau_n\le T)=1\,.
\end{equation}
\end{lemma}
\begin{proof}
  Рассуждения аналогичны тем, которые использовались при
доказательстве леммы \ref{le:traject}. 
Покажем, что для достаточно больших $T$
\begin{equation}
  \label{eq:51}
  \inf_{q:\,\sup_{t\le T}\sum_{k=1}^K q_k(t)> A}\mathbf
  I(q)=\inf_{q:\,\sup_{t\le T}\sum_{k=1}^K q_k(t)\ge A}\mathbf
  I(q)=\theta^\ast\cdot r\,.
\end{equation}
Для $\epsilon>0$ обозначим через $q^{\epsilon,\ast}$ оптимальную
траекторию для попадания из начала координат в точку
$r(1+\epsilon)$\,, построенную в доказательстве теоремы
\ref{the:optimum}.
 Предполагается, что она продолжена в соответствии с
законом больших чисел после попадания в $r(1+\epsilon)$\,. 
Поскольку $q^{\epsilon,\ast}$ до момента попадания в $r(1+\epsilon)$ -- это  обращённая жидкостная
траектория двойственной системы, 
то существует $\hat T>0$ такое, что все 
траектории $q^{\epsilon,\ast}$, соответствующие $\epsilon\in[0,1]$\,, попадают в точку
назначения к моменту $\hat T$\,, см., напр. лемму 5.4 на с.143 в
Bramson \cite{Bra06}.
Мы имеем, что для $T\ge \hat T$
\begin{multline*}
  \theta^\ast\cdot r=\mathbf I(q^\ast)
=\inf_{q:\,\sum_{k=1}^K q_k(t)\ge A\text{ для
      некоторого } t>0}\mathbf
  I(q)
\le \inf_{q:\,\sup_{t\le  T}\sum_{k=1}^K q_k(t)\ge A}\mathbf
  I(q)\\\le \inf_{q:\,\sup_{t\le  T}\sum_{k=1}^K q_k(t)> A}\mathbf
  I(q)\le \inf_{q:\,\sup_{t\le  T}\sum_{k=1}^K q_k(t)\ge A(1+\epsilon)}\mathbf
  I(q)\le \mathbf
  I(q^{\epsilon,\ast})=\theta^\ast\cdot r(1+\epsilon)\,.
\end{multline*}
Устремляя $\epsilon$ к нулю, получаем \eqref{eq:51}.
Так как,
 в силу принципа больших уклонений,
\begin{align*}
      \limsup_{n\to\infty}\frac{1}{n}\,
\ln\mathbf P( \tau_n\le T)=\limsup_{n\to\infty}\frac{1}{n}\,
\ln\mathbf P(\sup_{t\le T}\frac{Q(nt)}{n}\,\ge A)\le
-\inf_{q:\,\sup_{t\le T}\sum_{k=1}^K q_k(t)\ge A}\mathbf
  I(q),\\
\liminf_{n\to\infty}\frac{1}{n}\,
\ln\mathbf P( \tau_n\le T)=\liminf_{n\to\infty}\frac{1}{n}\,
\ln\mathbf P(\sup_{t\le T}\frac{Q(nt)}{n}\,\ge A)\ge
-\inf_{q:\,\sup_{t\le T}\sum_{k=1}^K q_k(t)> A}\mathbf
  I(q)\,,
\end{align*} 
то \eqref{eq:22} вытекает из \eqref{eq:51}.

 Для доказательства \eqref{eq:16} 
обозначим $\tau(q)=\inf\{t:\,\sum_{k=1}^Kq_k(t)\ge A\}$\,, где 
$q=(q(t)\,,t\ge0)\in\mathbb D(\R_+,\R_+^K)$\,. Так как множество
$\{q\in\mathbb D(\R_+,\R_+^K):\,\sup_{(\tau(q)-T^\ast)^+\le t\le
  \tau(q)}
\abs{q(t)- q^\ast (t-(\tau(q)-T^\ast)^+)}
\ge \delta,\,\sup_{t\le
  T}\sum_{k=1}^K q_k(t)\ge A\}$ содержит те свои предельные точки,
которые являются непрерывными функциями, и так как $\mathbf
I(q)=\infty$\,, если $q$ разрывна, то
\begin{align*}
 \limsup_{n\to\infty}\frac{1}{n}\,
\ln\mathbf P(\sup_{(\tau_n-T^\ast)^+\le t\le \tau_n}\abs{\frac{Q(nt)}{n}-
    q^\ast(t-(\tau_n-T^\ast)^+)}\ge \delta,\, \tau_n\le T)\\\le
-\inf_{\substack{q:\,\sup_{(\tau(q)-T^\ast)^+\le 
t\le \tau(q)}\abs{q(t)- q^\ast (t-(\tau(q)-T^\ast)^+)}\ge
  \delta,\\
 \sum_{k=1}^K q_k(t)\ge A\text{ для некоторого }t\in[0,T]}}\mathbf
  I(q)\,.
\end{align*}
Так как множество, по которому берётся последний $\inf$\,, 
замкнуто, то эта нижняя грань
 достигается некоторой функцией
 $\breve
q$\,. 
Как мы видели, если $\inf_{(q,T'):\,\sum_{k=1}^K q_k(T')\ge
  A}\int_0^{T'}L(q(t),\dot q(t))\,dt$ 
достигается в  точке $(q,T')$, то  $T'\ge T^\ast $ и
$q(T'-t)=q^\ast(T^\ast -t)$ при $0\le t\le T^\ast$ (как следствие, 
$\tau(q)=T'$). Следовательно, если  $\tau(\breve q)<T^\ast$, то пара $(\breve
q,\tau(\breve q))$ не достигает указанной нижней грани. Педположим, что
$\tau(\breve q)\ge T^\ast$. Тогда существует момент $t\in[\tau(\breve
q)-T^\ast,\tau(\breve q)]$, такой что  $\breve q(t)- q^*(t-\tau(\breve
q)+T^\ast)\ge\delta$, т.е. $(\breve q,\tau(\breve q))$ также не
достигает указанной нижней грани.
Таким образом, 
$\mathbf I(\breve q)\ge \int_0^{ \breve q}L(\breve q(t),
\dot{\breve q}(t))\,dt>\int_0^{T^\ast}L(q^\ast(t),\dot{
  q}^\ast(t))\,dt=
\theta^\ast\cdot r$\,. 
Поэтому
\begin{equation*}
      \limsup_{n\to\infty}\frac{1}{n}\,
\ln\mathbf P(\sup_{(\tau_n-T^\ast)^+\le t\le \tau_n}\abs{\frac{Q(nt)}{n}-
    q^\ast(t-(\tau_n-T^\ast)^+)}\ge \delta,\, \tau_n\le T)<-\theta^\ast\cdot r\,.
\end{equation*}
Сопоставляя это неравенство с \eqref{eq:22}, получаем, что для всех
достаточно больших $T$
\begin{equation*}
    \limsup_{n\to\infty}\frac{1}{n}\,
\ln\mathbf P(\sup_{(\tau_n-T^\ast)^+\le t\le \tau_n}\abs{\frac{Q(nt)}{n}-
    q^\ast(t-(\tau_n-T^\ast)^+)}\ge \delta| \tau_n\le T)<0\,,
\end{equation*}
что влечёт за собой \eqref{eq:16}.
\end{proof}
Есть примечательное объяснение тому, что $ q^\ast $ получается обращением
времeни из $\overline q$\,.
Как и в Ananatharam, Heidelberger,  Tsoukas \cite{AnaHeiTso90},
см. также Shwartz, Weiss\cite{ShwWei93}, 
заметим, что 
\begin{multline*}
\mathbf P(\sup_{t\le  T^\ast }\abs{\frac{Q(nt)}{n}- q^\ast (t)}\ge \delta|
\,\abs{\frac{Q(n T^\ast )}{n}-r}\le\epsilon)=\\=
\mathbf P(\sup_{t\le  T^\ast }\abs{\frac{
    Q(n( T^\ast -t))}{n}-  q^\ast ( T^\ast -t)}\ge \delta|
\,\abs{\frac{ Q(n T^\ast )}{n}-r}\le\epsilon)=  \\
=\mathbf P(\sup_{t\le  T^\ast }\abs{\frac{
\overline    Q^n(nt)}{n}-\overline q(t)}\ge \delta|
\,\abs{\frac{\overline Q^n(0)}{n}-r}\le\epsilon) \,,
\end{multline*}
где $\overline    Q^n(t)=Q(n T^\ast -t)$\,.
Если процесс $Q(t)$ является стационарным, то процесс $\overline
Q^n$ также является стационарным и представляет собой процесс длин
очередей в обращённой по времени сети Джексона. В частности, его
распределение не зависит от $n$\,. Так как $\overline q(t)$ является
жидкостным пределом эргодической сети Джексона, то стандартными
методами можно показать, что последняя вероятность сходится к нулю при
$n\to\infty $ и $\epsilon\to0$\,. 
Похожие рассуждения можно найти в
Majewski, Ramanan \cite{MajRam08}, а также в Collingwood
\cite{Col15}, где стационарность также играет ключевую роль.
(Заметим, что ни одна из этих работ не опубликована в рецензируемых журналах.)

Так как $\theta^\ast \cdot r=\inf_{\substack{q:\,q(0)=0\,, q(T)=r\\ \text{ для
    некоторого } T}}\mathbf I(q)$\,, то, в соответствии с общей
теорией, см., Вентцель, Фрейдлин \cite{wf-r}, Shwartz, Weiss
\cite{SchWei95}, а также исходя из второго утверждения леммы \ref{le:traject}, можно предположить следующую асимптотику стационарного
распределения $Q(t)$\,:
\begin{equation*}
  \lim_{\epsilon\to0}\limsup_{n\to\infty}\frac{1}{n}\,\ln \mathbf 
P(\abs{\frac{Q(nt)}{n}\,-r}<\epsilon)=
  \lim_{\epsilon\to0}\liminf_{n\to\infty}\frac{1}{n}\,\ln \mathbf
  P(\abs{\frac{Q(nt)}{n}\,-r}<\epsilon)=-\theta^\ast  \cdot r\,.
\end{equation*} В
данном случае это свойство, очевидно, выполнено, т.к. стационарное
распределение $Q(t)$ известно явно:
$\mathbf P(Q(t)=(i_1,\ldots,i_K))=\prod_{k=1}^K
(1-\rho_k)\rho_k^{i_k}$\,, см., напр., Клейнрок \cite{Kle79-r}.
\section{Пример:  эргодическая сеть из двух узлов}
\label{sec::--primer}

Проиллюстрируем общие результаты на примере эргодической сети из двух узлов. В этом
случае анализ удаётся сделать более наглядным.
Имеем, что \begin{align}
  \label{eq:15'}
  h_1(\theta)&=e^{-\theta_1}\bl((e^{\theta_1}-1)p_{11}+(e^{\theta_2}-1)p_{12}+1\br)-1\,,\\\label{eq:15a}
  h_2(\theta)&=e^{-\theta_2}\bl((e^{\theta_1}-1)p_{21}+(e^{\theta_2}-1)p_{22}+1\br)-1\,\,,\\
\notag  H_1(\theta)&=(e^{\theta_1}-1)\lambda_1+(e^{\theta_2}-1)\lambda_2+
h_1(\theta)^+\mu_1+h_2(\theta)\mu_2,\\
\label{eq:15b}
  H_2(\theta)&=(e^{\theta_1}-1)\lambda_1+(e^{\theta_2}-1)\lambda_2+
h_1(\theta)\mu_1+h_2(\theta)^+\mu_2,
\\\notag
 H_0(\theta)&=(e^{\theta_1}-1)\lambda_1+(e^{\theta_2}-1)\lambda_2+h_1(\theta)\mu_1
+h_2(\theta)\mu_2\,.
\end{align}
Заметим, что  $h_1(\theta)=0$ тогда и только тогда, когда
$(e^{\theta_1}-1)/(e^{\theta_2}-1)=p_{12}/(1-p_{11})$\,,
т.е. $\theta_1$ и $\theta_2$ имеют один и тот же знак, следовательно,
график $h_1(\theta)=0$ лежит в I--м и III--м квадрантах. Кроме того,
$h_1(\theta)>0$ над графиком и слева от него и 
$h_1(\theta)<0$ под графиком и справа от него. Задавая $\theta_2$ как
функцию $\theta_1$ в первом квадранте уравнением $h_1(\theta)=0$\,,
 имеем, дважды дифференцируя, что $d^2\theta_2/d\theta_1^2=
(1-p_{11})/p_{12}\,e^{\theta_1-\theta_2}-((1-p_{11})/p_{12}\,e^{\theta_1-\theta_2})^2$\,. 
Так как $(1-p_{11})/p_{12}\,e^{\theta_1-\theta_2}=1+
(1-p_{11}-p_{12})/p_{12}e^{-\theta_2}\ge 1$\,, то $\theta_2$ является
строго вогнутой функцией $\theta_1$\,, если $p_{11}+p_{12}<1$ и
является линейной функцией $\theta_1$\,, если  $p_{11}+p_{12}=1$\,.
 Аналогичнo,
 график $h_2(\theta)=0$ лежит в I--м и III--м квадрантах,
$h_2(\theta)<0$ над графиком и слева от него,
$h_2(\theta)>0$ под графиком и справа от него,  $\theta_2$
 является строго выпуклой функцией $\theta_1$ на графике $h_2(\theta)=0$\,,  если $p_{22}+p_{21}<1$ и
является линейной функцией $\theta_1$\,, если  $p_{22}+p_{21}=1$\,.
Так как $(1-p_{11})/p_{12}\ge p_{21}/(1-p_{22})$\,, то в первом
квадранте
график $h_1(\theta)=0$ лежит над графиком $h_2(\theta)=0$ и лежит строго над, если 
$(1-p_{11})/p_{12}> p_{21}/(1-p_{22})$\,.
В третьем квадранте
график $h_1(\theta)=0$ лежит под графиком $h_2(\theta)=0$ и лежит
строго 
под, если 
$(1-p_{11})/p_{12}> p_{21}/(1-p_{22})$\,.

Найдём оптимальное движение из  $(0,0)$ в $(1,0)$ по прямой.
Движение происходит в соответствии с гамильтонианом $H_2(\theta)$ при
дополнительном условии $q_2(t)=0$\,. Тогда, по лемме \ref{le:lagrange}, выбирая
$g(q)=q_2^2$\,, так что $\gamma(t)=0$\,, имеем, что
$(\dot q_1(t),0)=(\dot q_1(t),\dot q_2(t))\in\partial
H_2(\theta)$\,. Как следствие, $0 \in\partial_2H_2(\theta)$ и
$\dot q_1(t)\in\partial_1 H_2(\theta)$\,, см., Предложение 2.3.15 на
с.52 в Кларк \cite{Cla83-r}.
Кроме того,  $H_2(\theta)
=0$\,.
 Затраты равны  $1\cdot
\theta_1+0\cdot\theta_2=
\theta_1$\,.
В силу \eqref{eq:15'}, \eqref{eq:15a} и \eqref{eq:15b},
для того, чтобы $\partial_{2}H_2(\theta)\ni0$\,,
 необходимо,
чтобы 
$h_2(\theta)\ge0\,$ (т.к. $\partial_2 h_1(\theta)>0$\,, см., напр., 
\eqref{eq:34}). 

\begin{figure}[h!]
  \centering
  \includegraphics[scale=.7]{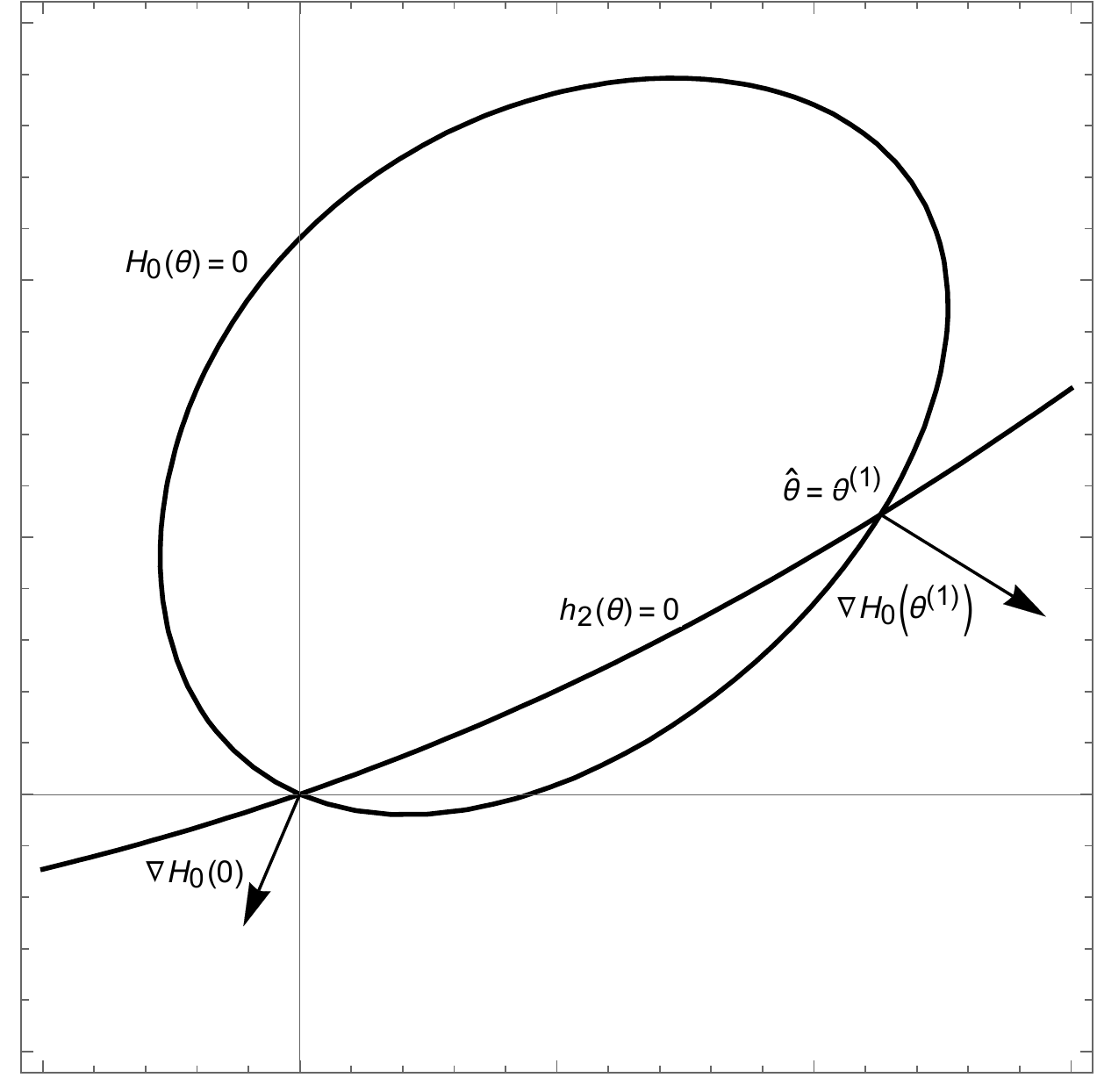}
  \caption{Случай, когда 
$\partial_2 H_0(\theta^{(1)})\le 0$\,}
      \label{fig:5}
\end{figure}
\begin{figure}[h!]
  \centering
  \includegraphics[scale=.7]{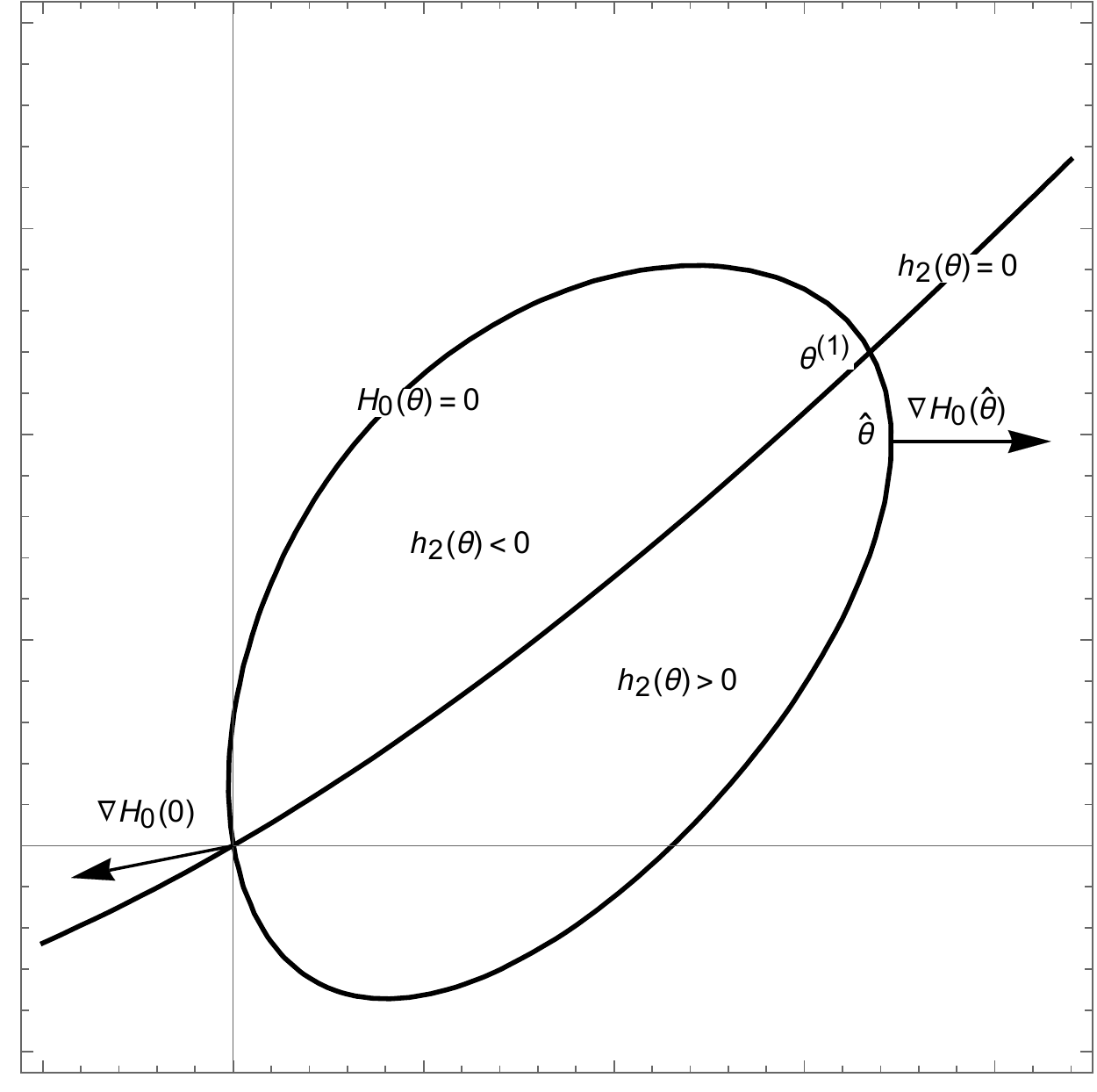}
  \caption{Случай, когда 
$\partial_2 H_0(\theta^{(1)})> 0$\,}
      \label{fig:5}
\end{figure}
Заметим также, что если $h_2(\theta)=0$\,, то
\begin{equation*}
    \partial H_2(\theta)=
    \Big\{
  \Bl(\begin{array}[c]{c}
    e^{\theta_1}\lambda_1+\partial_1h_1(\theta)\mu_1
+\alpha\partial_1h_2(\theta)\mu_2\\e^{\theta_2}\lambda_2+
\partial_2h_1(\theta)\mu_1 +\alpha\partial_2h_2(\theta)\mu_2
  \end{array}\Br)\,,\alpha\in[0,1]\Big\}\,.
\end{equation*}
Таким образом,  если точка $\hat\theta$ такова, что $H_0(\hat\theta)=0$\,,
$h_2(\hat\theta)=0$\,, т.е., $\hat\theta=\theta^{(1)}$\,, и
$\partial_2 H_0(\theta^{(1)})\le 0$\,,
 то эта точка оптимальна. Если $h_2(\hat\theta)>0$\,, $H_0(\hat\theta)=0$
  и $\partial_2 H_0(\hat\theta)=0$\,, то $\hat\theta$ также будет
  оптимальной. Как можно видеть, эти две возможности исключают друг друга.
Они иллюстрируются  рис. 1 и рис. 2, соответственно. (Вектор $\nabla
H_0(0)$ изображен по той причине, что принадлежность этого вектора
третьему квадранту является достаточным условием  эргодичности сети.)

Рассмотрим оптимальные пути попадания из $(0,0)$ в $r=(r_1,r_2)$\,.  
Сравним  путь по прямой и путь горизонтально  и затем по прямой.
Покажем, что в случае, когда движение сначала происходит горизонтально
по некоторому вектору $s=(s_1,0)$ с
импульсом $\hat\theta=(\hat\theta_1,\hat\theta_2)$  таким, что 
 $\partial_2H_0(\hat\theta)=0$ и $h_2(\hat\theta)>0$\,, как на рис.2, то непосредственное 
 движение по вектору $r$
 не хуже. Рис. 3 иллюстрирует приводимые рассуждения.
\begin{figure}[h!]
  \centering
  \includegraphics[scale=.7]{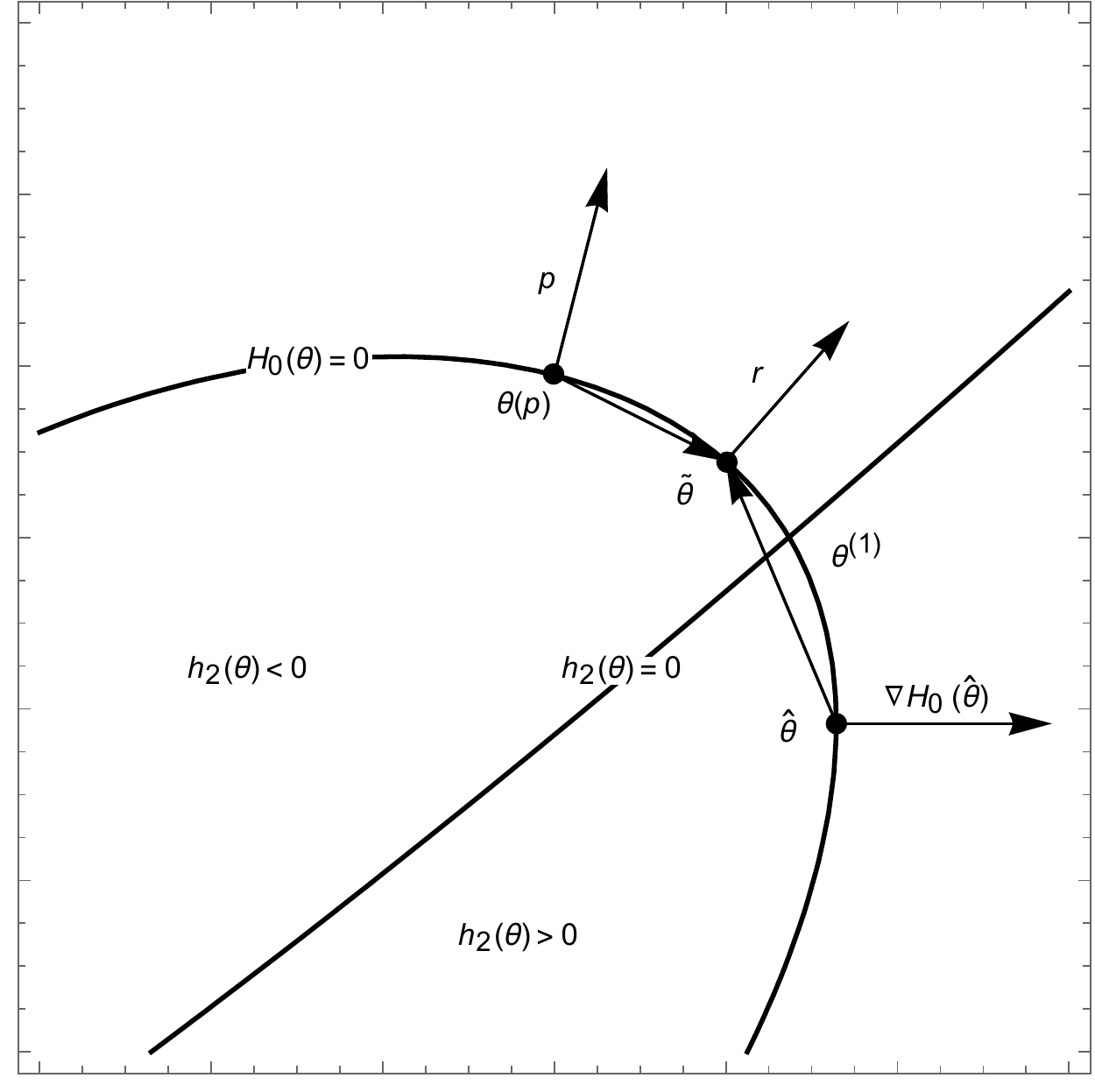}
  \caption{Предпочтительность  движения по вектору $r$}
      \label{fig:5}
\end{figure}
 Затраты на движение
 сначала по вектору $s$ 
и затем по вектору $p=r-s$  равны $\hat\theta_1s_1+
\theta(p) \cdot p$\,, где через $\theta(p)$ обозначена точка на кривой
$H_0(\theta)=0$\,, в которой внешняя нормаль совпадает с вектором
$p$\,. 
Обозначим также через $\tilde \theta$ точку на кривой $H_0(\theta)=0$ с внешней
нормалью $r$\,. Если $r$ не лежит на оси абсцисс, то $\tilde
\theta\not=\hat\theta$\,, поскольку движение с $\hat\theta$ горизонтально. 
  Затраты на движение сразу по вектору $r$ равны
$ \tilde\theta \cdot r= \tilde\theta \cdot s+ \tilde\theta \cdot p
$\,.
Так как $s$ ортогонален кривой $H_0(\theta)=0$ в точке $\hat\theta$
и функция $H_0(\theta)$ является строго выпуклой,
то  $ (\tilde\theta-\hat\theta)\cdot s\le0$\,, причем неравенство является
строгим, если $r$ не лежит на оси абсцисс и $s\not=0$\,. Аналогично, так как
 $p$ ортогонален кривой $H_0(\theta)=0$ в точке $\theta(p)$\,, то 
$(\tilde\theta-\theta(p))\cdot p\le0$\,. Таким образом,
$\tilde\theta \cdot r\le \theta \cdot s+
\theta(p)\cdot p$\,, где  неравенство является строгим если $r$ не лежит на
оси абсцисс, т.е.,
 затраты при движении  по негоризонтальному вектору
$r$ меньше затрат при первоначальном движении по $s$\,.

Как и выше,  пусть
$\theta^{(1)}$ -- точка пересечения кривых $H_0(\theta)=0$ и
$h_2(\theta)=0$ и пусть $\theta^{(2)}$ -- точка пересечения кривых
$H_0(\theta)=0$ и $h_1(\theta)=0$\,. 
Как  доказано в теореме \ref{the:H}, прямые $\theta_1=\theta_1^{(1)}$ и 
$\theta_2=\theta_2^{(2)}$ пересекаются в точке $\theta^\ast $\,,
которая лежит на кривой $H_0(\theta)=0$\,. 
 Приведенный выше анализ показывает, что если
$\partial_2H_0(\theta^{(1)})>0$\,, то начальное движение по горизонтали не является
оптимальным. Более того, как видно из дальнейшего, в этом случае движение
вертикально и затем пересечение внутренности квадранта так, чтобы
оказаться в нужном месте оси абсцисс стоит меньше даже в случае,
когда требуется переместиться горизонтально.
  Аналогично, начальное вертикальное движение не является
оптимальным, если $\partial_1H_0(\theta^{(2)})>0$\,.

\begin{figure}[h!]
  \centering
  \includegraphics[scale=.7]{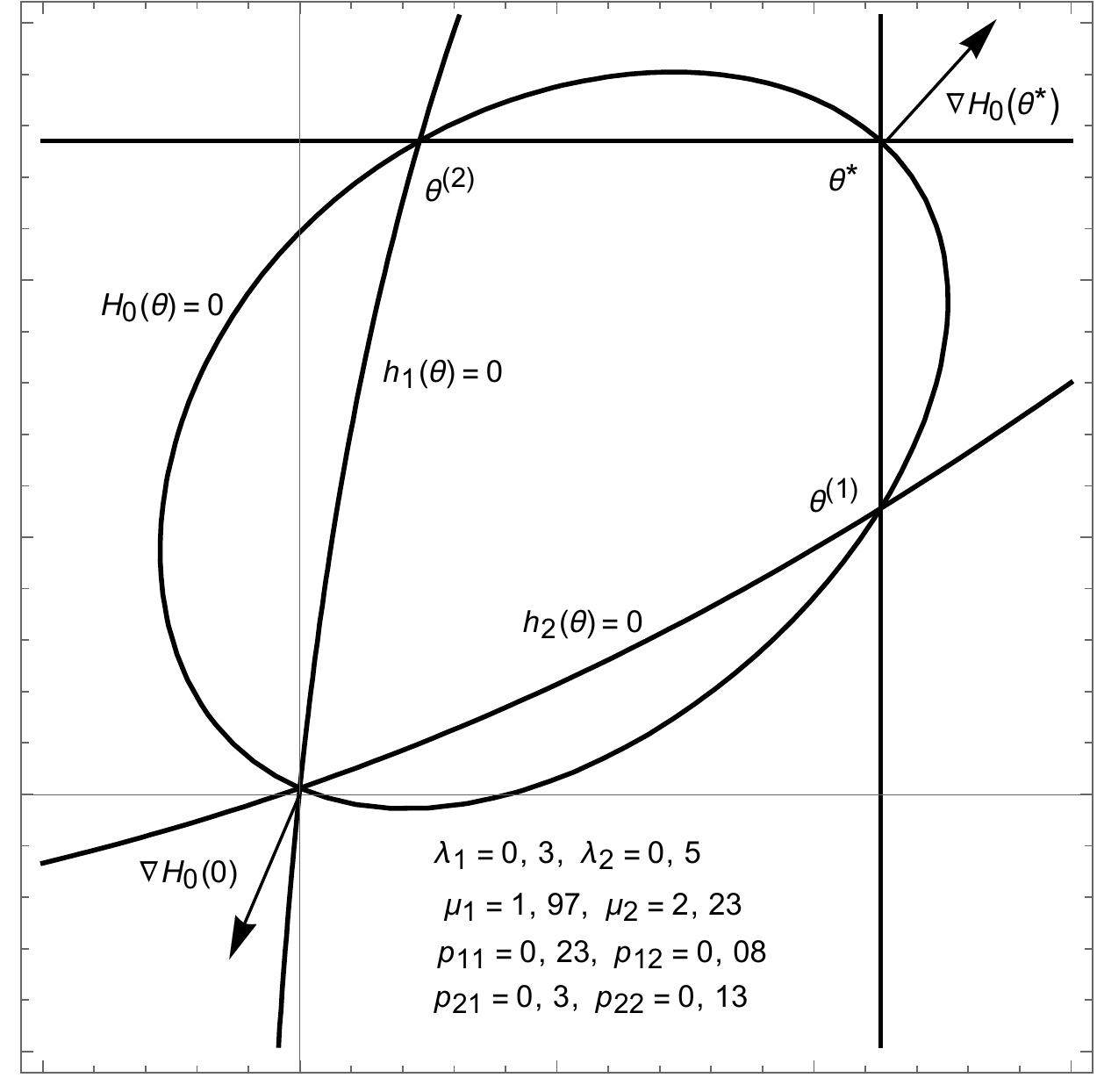}
  \caption{Движение внутри квадранта направо и вверх}
  \label{fig:1}
\end{figure}
Предположим, что 
$\partial_2H_0(\theta^{(1)})\le0$ и
$\partial_1H_0(\theta^{(2)})\le0$\,, как на рис. 4. 
Тогда
$\partial_1H_0(\theta^{(1)})>0$ и
$\partial_2H_0(\theta^{(2)})>0$\,.  В этом случае оптимальным
 горизонтальным движением является движение с импульсом
$\theta^{(1)}$
и оптимальным  вертикальным движением 
является движение с импульсом $\theta^{(2)}$\,.
Затраты на движение
 сначала по вектору $s=(s_1,0)$ горизонтально
 и затем по вектору $p$  равны $\theta^{(1)}_1s_1+
\theta(p)\cdot p$\,.
Поэтому, как и на рис. 3, если $\tilde\theta_1\le \theta^{(1)}_1$\,, что
имеет место, если 
наклон вектора $r$ больше, чем
наклон нормали к кривой $H_0(\theta)=0$ в точке $\theta^\ast $\,, и,
соответственно, 
$\tilde\theta$ находится слева от $\theta^\ast $\,, то
движение по $r$ предпочтительней.
 Предположим, что наклон вектора $r$ меньше, чем
наклон нормали к кривой $H_0(\theta)=0$ в точке $\theta^\ast $\,, т.е.,
$\tilde\theta$ находится справа от $\theta^\ast $\,. Найдём
оптимальное горизонтальное смещение $s$\,.
  Нужно минимизировать 
$\theta^{(1)}_1s_1+
\theta(p)\cdot p$ по $\theta(p)$ при
условии, что $r=p+s$  и $p$ пропорционален $\nabla H_0(\theta(p))$\,.
 Обозначая $\theta(p)$ через $\theta$ 
и обозначая $p$ через $p(\theta)$\,,
 имеем, что
$\theta^{(1)}_1s_1+\theta\cdot p(\theta)=\theta^{(1)}_1(r_1-p_1(\theta))+
\theta \cdot p(\theta)
=\theta^{(1)}_1r_1-\theta^{(1)}_1p_1(\theta)+
\theta\cdot p(\theta)$\,, где 
$\kappa \nabla H_0(\theta)=p(\theta)$ для некоторого $\kappa>0$\,. Так как $p_2(\theta)=r_2$\,, то
$\kappa \partial_2 H_0(\theta)=r_2$\,.
Минимизируем $-\theta^{(1)}_1p_1(\theta)+
\theta\cdot p(\theta)$ при ограничениях $H_0(\theta)=0$\,,
$\kappa \nabla H_0(\theta)=p(\theta)$\,, 
$\kappa \partial_2 H_0(\theta)=r_2$\,.
Имеем, что $-\theta^{(1)}_1p_1(\theta)+
\theta \cdot p(\theta)=\bl(
 \partial_1 H_0(\theta)/\partial_2  H_0(\theta)
\,(\theta_1-\theta^{(1)}_1)+\theta_2\br)r_2$\,. 
 Покажем, что
$ \partial_1 H_0(\theta)/\partial_2
  H_0(\theta)\,(\theta_1-\theta^{(1)}_1)+\theta_2$
убывает на дуге $[\tilde\theta,\theta^\ast ]$\,.
 Предположим, что
$\theta$ находится между $\tilde\theta$ и $\theta^\ast $\,.
\begin{figure}[h!]
  \centering
  \includegraphics[scale=.8]{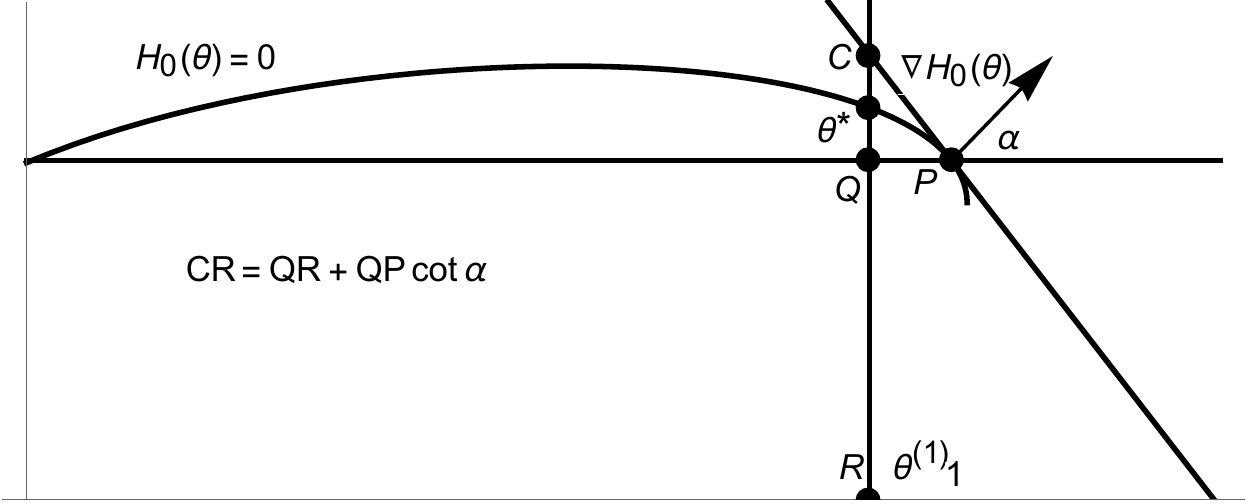}
  \caption{Оптимальность движения по нормали в $
\theta^\ast$}
      \label{fig:5}
\end{figure}
Рис. 5 иллюстрирует приводимые ниже рассуждения. Точка $P$
сoответствует переменному импульсу $\theta$\,.
Обозначим через $\alpha$ угол между нормалью к $ H_0(\theta)=0$ в
точке $P$ и
горизонталью
 через точку $\theta$\,. Имеем, что $\partial_1 H_0(\theta)/\partial_2
  H_0(\theta)=\cot \alpha$\,.  По свойству углов с
взаимно перпендикулярными сторонами, угол $QCP$ между касательной к
$ H_0(\theta)=0$ в  точке $\theta$ и вертикальной прямой с абсциссой
$\theta^{(1)}_1$ также равен $\alpha$. Поэтому отрезок $CQ$\,,
соединяющий точку $(\theta_1^{(1)},\theta_2)$ и пересечение касательной
с вертикальной прямой с абсциссой $\theta^{(1)}_1$\,, равен
$(\theta_1-\theta^{(1)}_1) \cot\alpha$\,. Поэтому 
$ \partial_1 H_0(\theta)/\partial_2
H_0(\theta)\,(\theta_1-\theta^{(1)}_1)
+\theta_2$ -- это длина вертикального
   отрезка $CR$ от точки пересечения касательной
с вертикальной прямой с абсциссой $\theta^{(1)}_1$ до оси
абсцисс. По мере того, как точка $\theta$ движется по кривой 
$ H_0(\theta)=0$ от точки $\tilde\theta$ против часовой стрелки, 
длина этого
отрезка уменьшается и достигает минимума, когда $\theta$ попадает в
точку $\theta^\ast $\,. Если продолжать движение 
за $\theta^\ast $ против часовой
стрелке, то отрезок опять начнет увеличиваться.
Таким образом, если начать двигаться горизонтально, то
 оптимальным является движение с импульсом $\theta^{(1)}$ до тех пор, пока
наклон прямой, соединяющей движущуюся точку и точку $r$ не станет
равным наклону внешней нормали в точке $\theta^\ast $\,. После этого
нужно двигаться прямолинейно с  наклоном
 этой внешней нормали. Поскольку движение по $r$ является частным
 случаем начального движения горизонтально с $s_1=0$\,,
в этом случае оптимальное
 движение горизонтально  лучше, чем  движение по $r$\,.  
Наконец, если сначала
двигаться вертикально и затем по направлению к точке $r$\,, то,
поскольку $\tilde\theta_2<  \theta^\ast_2$\,,  ситуация будет
аналогична рассмотренному выше случаю, когда
$\tilde\theta_1<\theta^\ast _1$\,, т.е., сразу двигаться по прямой
выгоднее. 
Таким образом, 
оптимальным является движение по  оси абсцисс  до того момента,
пока прямая, соединяющая движущуюся точку  и точку $r$\,,
 не будет иметь
тот же наклон, что и нормаль к $ H_0(\theta)=0$ в точке
$\theta^\ast $\,. 
После этого оптимальным является движение по направлению этой нормали.
При этом 
$s_1=r_1-p_1=r_1-\partial_1 H_0(\theta^\ast )/\partial_2 H_0(\theta^\ast )\,r_2
=r_2(r_1/r_2-\partial_1 H_0(\theta^\ast )/\partial_2 H_0(\theta^\ast ))$\,.
Затраты равны $\theta^{(1)}_1s_1+\theta^\ast \cdot(r-s)=\theta^\ast \cdot
r$\,.
Аналогично, если наклон вектора $r$ больше, чем наклон нормали в точке $\theta^\ast $\,, 
то нужно сначала двигаться вертикально, пока наклон не
сравняется с наклоном нормали, и затем идти по нормали в $\theta^\ast $\,.
Затраты также равны $\theta^\ast \cdot r$\,.
Эти варианты иллюстрируются рис. 6.
\begin{figure}[h!]
  \centering
  \includegraphics[scale=.5]{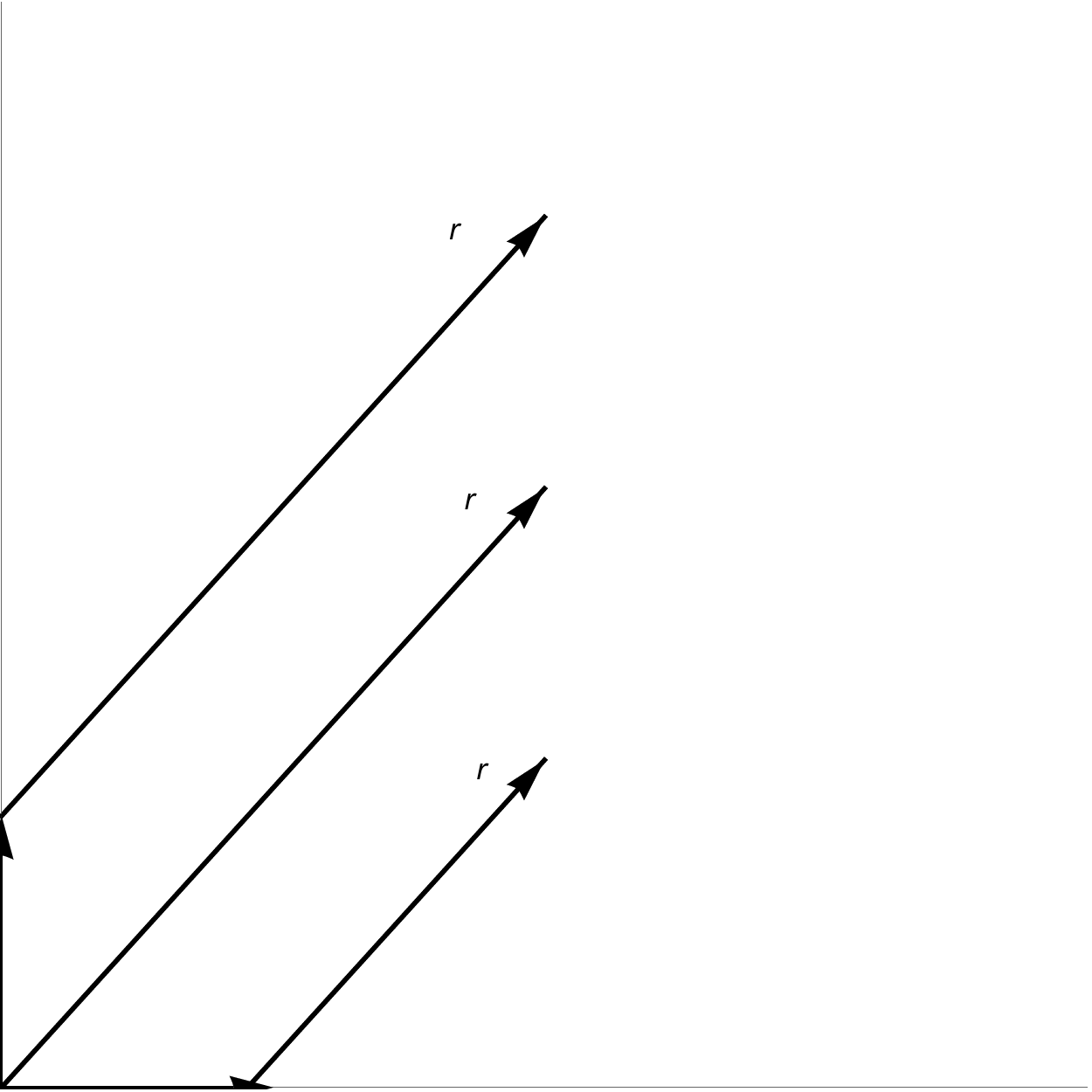}
  \caption{Оптимальные траектории для параметров рис. 4 в зависимости от конечной точки}
  \label{fig:10}
\end{figure}

\begin{figure}[!h]
  \centering
  \includegraphics[scale=.7]{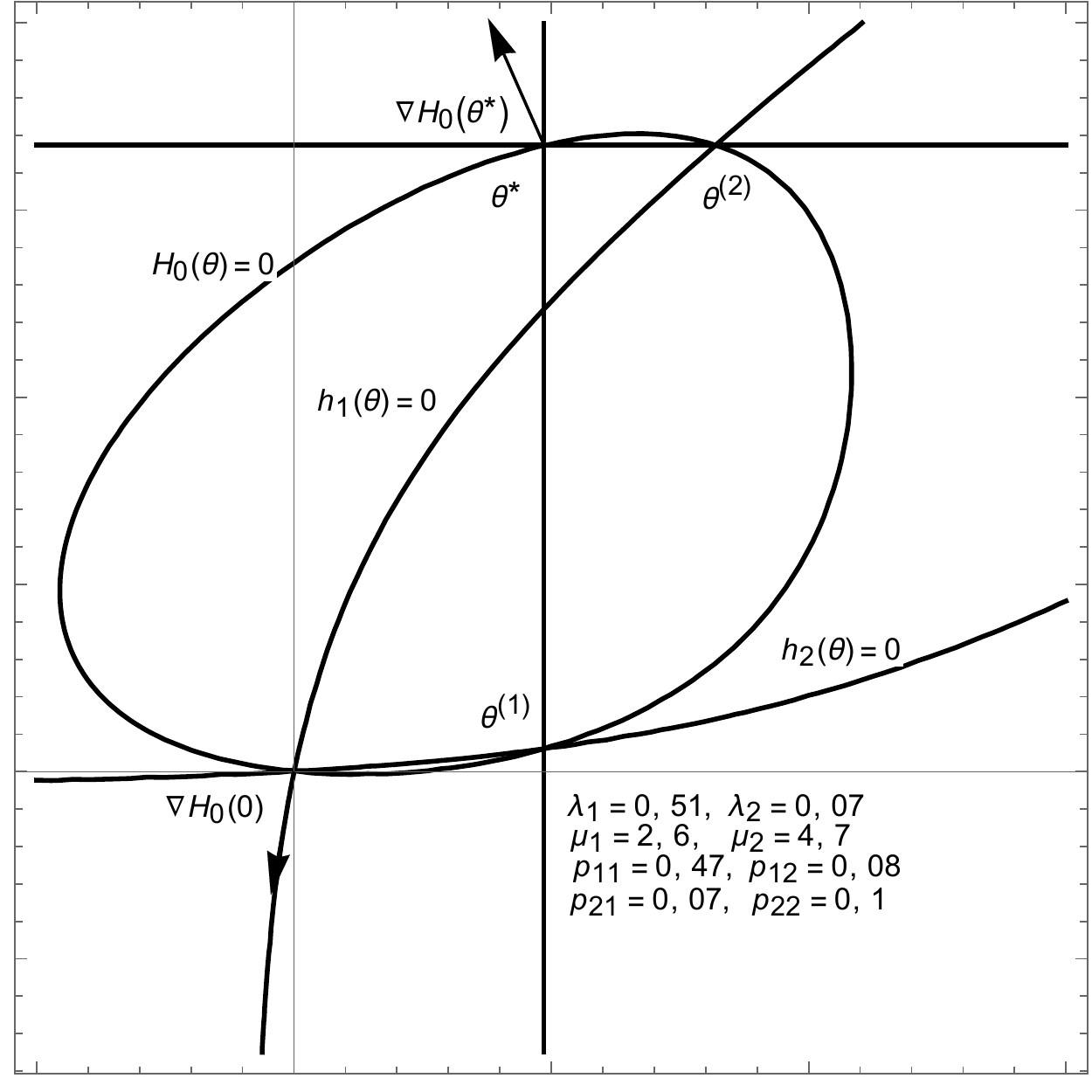}
  \caption{Движение внутри квадранта налево и вверх}
  \label{fig:2}
\end{figure}
\begin{figure}[!h]
  \centering
  \includegraphics[scale=.7]{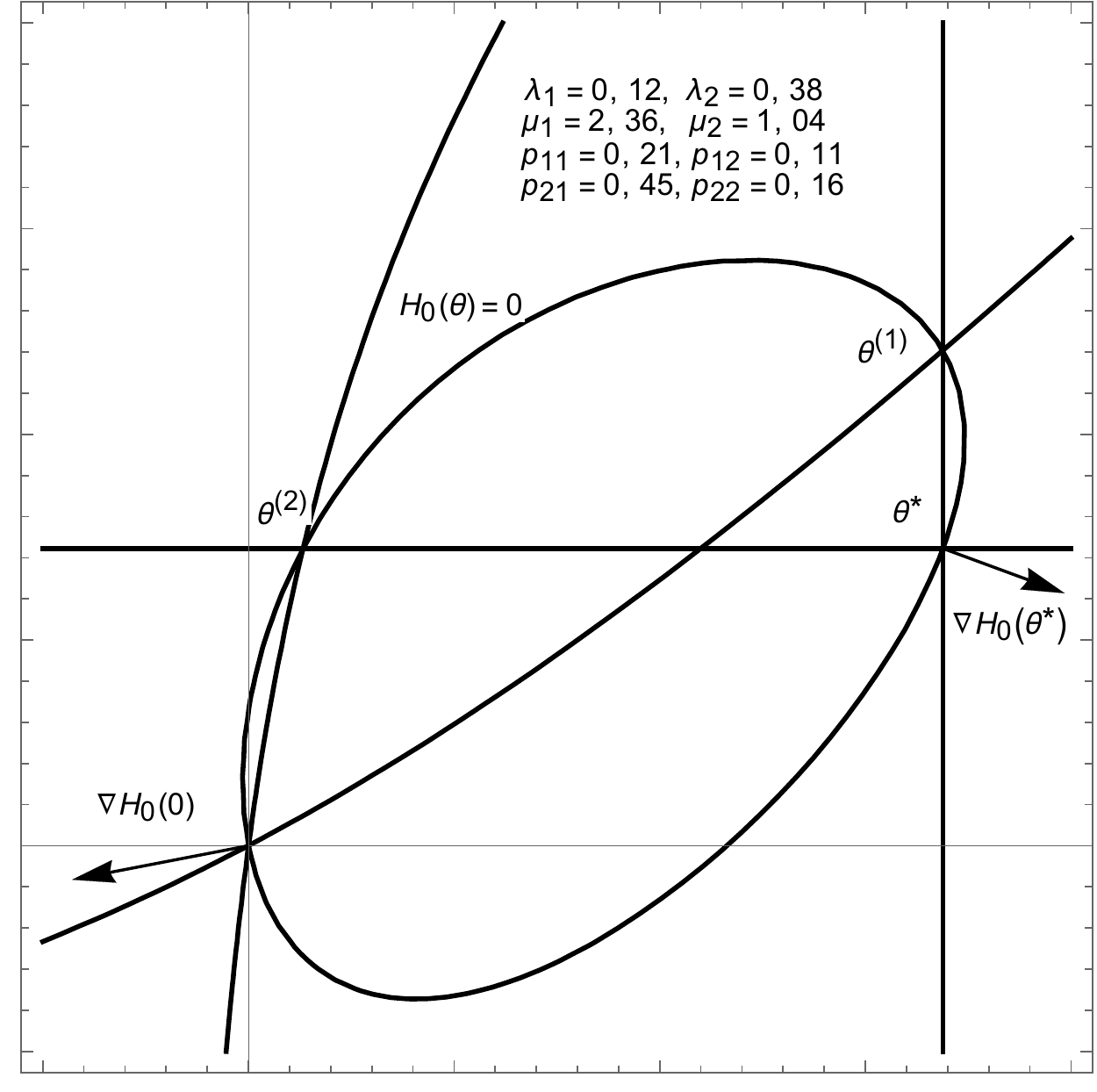}
  \caption{Движение внутри квадранта направо и вниз}
  \label{fig:3}
\end{figure}
\begin{figure}[!h]
  \centering
  \includegraphics[scale=.5]{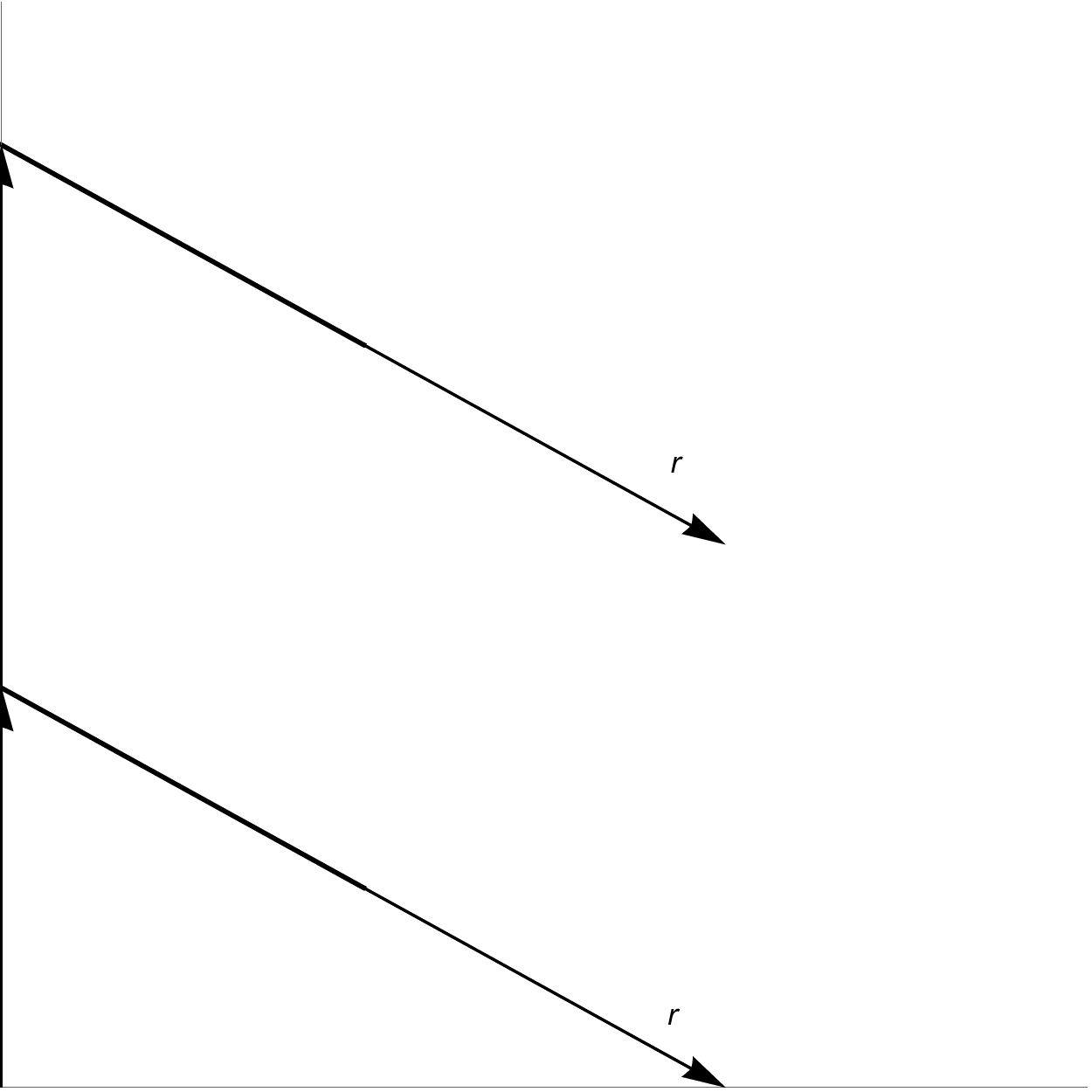}
  \caption{Оптимальные траектории при параметрах рис. 8}
  \label{fig:11}
\end{figure}

\begin{figure}[!h]
  \centering
  \includegraphics[scale=.7]{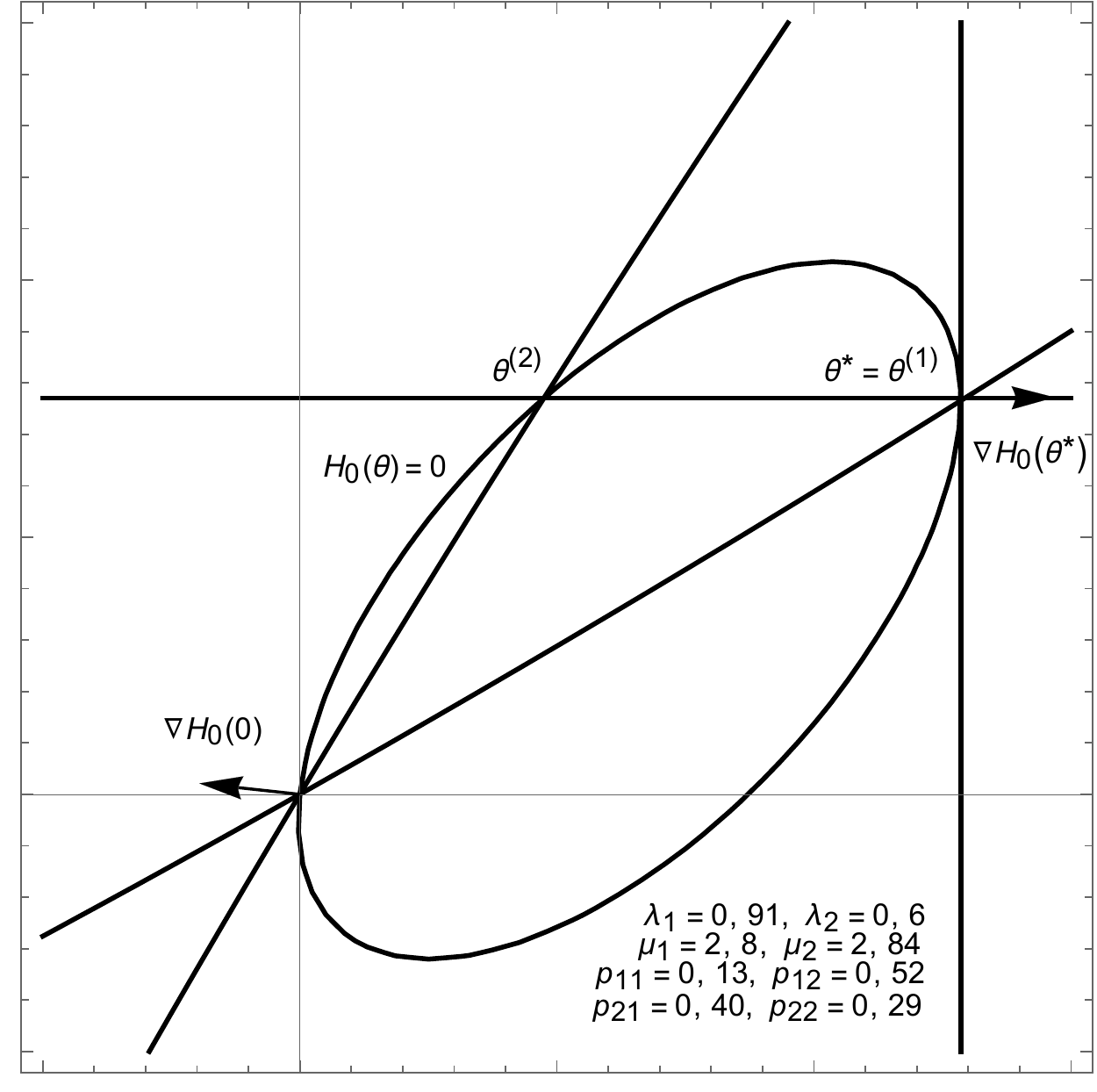}
  \caption{Горизонтальное движение внутри квадранта}
  \label{fig:4}
\end{figure}

\begin{figure}[!h]
  \centering
  \includegraphics[scale=.7]{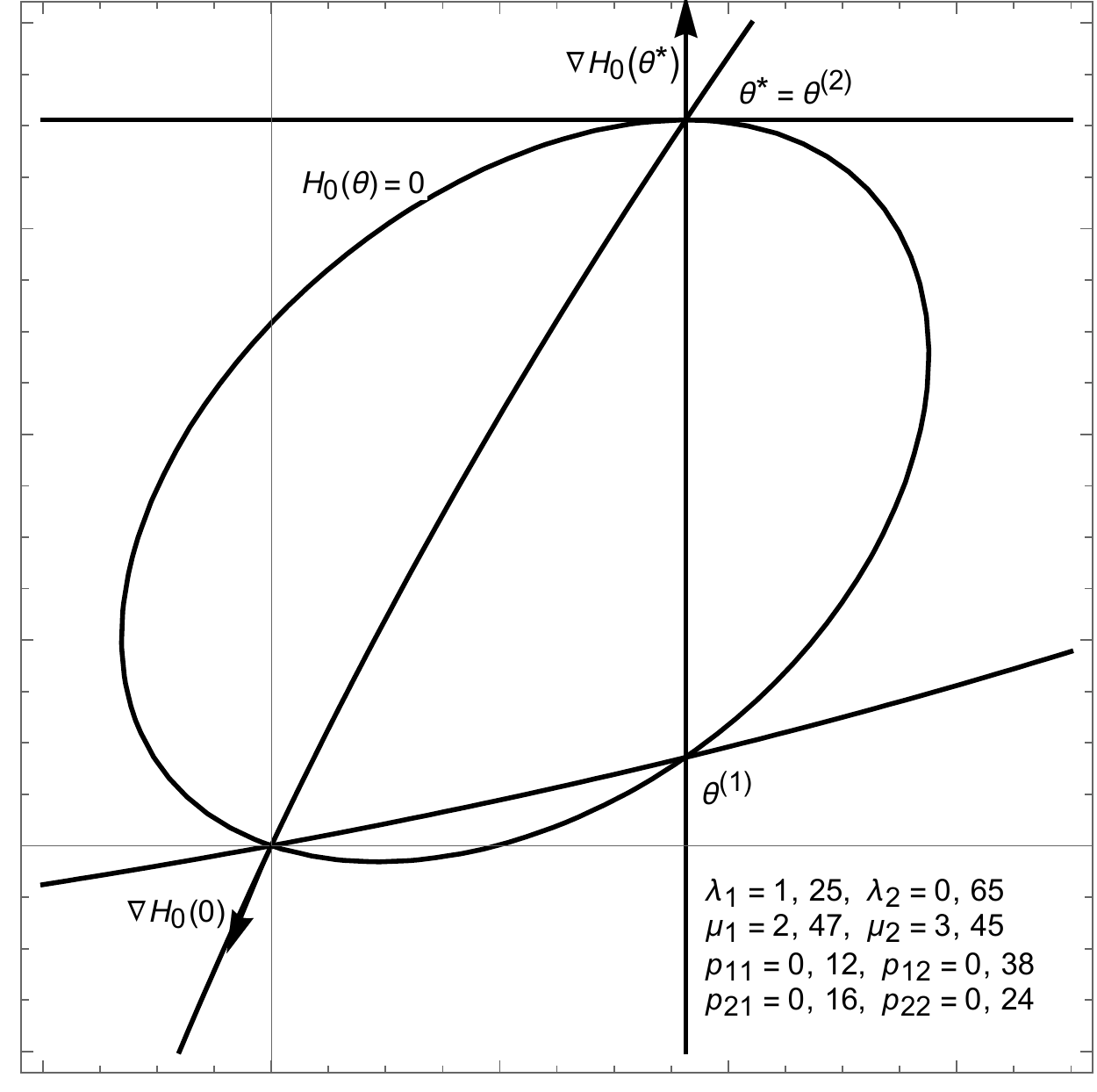}
  \caption{Вертикальное движение внутри квадранта}
  \label{fig:5}
\end{figure}

Если  
$\partial_2H_0(\theta^{(1)})\le0$
и $\partial_1H_0(\theta^{(2)})>0$\,, то точка $\theta^\ast $ находится
левее точки $\theta^{(2)}$\,, см. рис 7.
 Так как оптимальное вертикальное движение будет
происходить в соответствии с апексом кривой $H_0(\theta)=0$\,, то оно будет
уступать движению прямо по направлению $r$\,. 
Геометрические рассуждения аналогичные использованным выше
показывают, что в этом случае нужно двигаться
горизонтально, пока наклон не сравняется с наклоном внешней нормали в
точке $\theta^\ast $ и после этого ''возвращаться'' по этой нормали.
Если  
$\partial_2H_0(\theta^{(1)})>0$ и
$\partial_1H_0(\theta^{(2)})\le 0$\,,  см. рис 8, то нужно двигаться вертикально и
потом ''возвращаться''. Рис. 9 иллюстрирует оптимальные траектории для
параметров рис. 8. Ось абсцисс при этом
является несущественной, так что оптимальная траектория, ведущая
 в точку на этой оси,
проходит по оси ординат.
Во всех случаях надо двигаться по одной из осей координат так, чтобы
вектор, соединяющий движущуюся точку с точкой назначения, стал
колинеарен нормали в точке $\theta^\ast $\,.  После этого нужно
двигаться по прямой к точке назначения. 
Если система эргодична, то одна из этих комбинаций знаков 
$\partial_2 H_0(\theta^{(1)})$ и $\partial_1
H_0(\theta^{(2)})$ обязательно имеет место.
Затраты всегда равны $\theta^\ast \cdot r$\,.
Как следствие, если $\partial_2 H_0(\theta^{(1)})>0$
(соответственно, $\partial_1 H_0(\theta^{(2)})>0$), то первоначальное
движение по горизонтали (соответственно, по вертикали) не является
частью оптимального пути, т.е., ось абсцисс (соответственно, ось ординат)
является несущественной.  Кроме того, оптимальный маршрут можно получить,
двигаясь по антиградиенту $H_0$ в точке $\theta^\ast $
 до пересечения с одной из
осей координат и затем двигаясь по этой оси в начало координат. Эта
ось обязательно является существенной.
Приводимые  рис. 10 и рис. 11 иллюстрируют  варианты 
   оптимального движения внутри первого квадранта параллельно одной из осей координат. (Хотя
   вектор $\nabla H_0(0)$ на рис. 10 не принадлежит третьему
   квадранту, выкладки показывают, что сеть, тем не менее, эргодична.)

{\bf Благодарность.}
Автор выражает признательность С.А. Пирогову и А.Н. Рыбко за полезные
обсуждения и рекомендации по улучшению изложения.
\appendix
\section{Приложение}
\label{sec:app}

\subsection{Доказательство леммы \ref{le:lagrange}}
\label{sec:--refle:lagrange}

  Eсли минимум 
$\int_0^TL(x(t),\dot x(t))\,dt$ по $T\ge0$ и по абсолютно непрерывным
 функциям $x(t)$ таким, что $x(0)=x_0$\,, $(T,x(T))\in S'$ и
$g(x(t))\le0$ для всех $t\le T$\,,
где  $S'$ -- выпуклое замкнутое подмножество $\R\times\R^k$\,,
   достигается для времени $T^\ast $ и функции $
x^\ast(t)$\,, то  существуют 
 мера $\mu$ на $[0,T^\ast ]$\,, $\mu$--измеримая функция
$\gamma(t)$\,,   абсолютно непрерывная функция $p(t)$
 такие, что 
\begin{enumerate}
\item $\gamma(t)\in \partial^>g(x^\ast(t))$ для $\mu$--почти всех
  $t\in[0,T^\ast]$ и носитель меры $\mu$ содержится в множестве
  $\{t\in[0,T^\ast]:\,\partial^>g( x^\ast(t))\not=\emptyset\}$\,,
  \item \begin{equation*}
\left[  \begin{array}[c]{c}
    -\dot p(t)\\\dot  x^\ast(t)
  \end{array}\right]\in \partial H\bl(x^\ast(t),p(t)+\int_0^t
\gamma(s)\, \mu(ds)\br)
\quad \text{п.в. на } [0,T^\ast]\,,
\end{equation*}
\item 
существует постоянная $h$ такая, что
  \begin{equation*}
    H(x^\ast(t),p(t)+\int_0^t\gamma(s)\, \mu(ds))=h\quad\text{  на
    } [0,T^\ast]\,,
  \end{equation*}
\item
  \begin{equation*}
    \left[
      \begin{array}[c]{c}
-h\\p( T^\ast  )+\displaystyle\int_0^{ T^\ast  } \gamma(t)\,\mu(dt)        
      \end{array}
\right]
 \in -N_{S'}(T^\ast, x^\ast(T^\ast))\,.
  \end{equation*}
\end{enumerate}
Несколько более подробно,  если множество $S'$ имеет вид $\{T'\}\times
S''$ для некоторых $T'>0$ и
выпуклого замкнутого множества $S''\subset \R^k$\,, то утверждения 1 и 2
   вытекают из теоремы 10.2.1 и
обсуждения на сс.362 -- 364 в Vinter \cite{vin00},  а также из ''гамильтоновой
дуализации'', см., теорему 7.6.5 на с.266 в Vinter \cite{vin00}. Для  общего
выбора $S'$ можно применить прием сведения задачи с нефиксированным
временем к задаче с фиксированным временем, как 
в доказательстве следствия 3.6.1 на с.142 в Кларк \cite{Cla83-r}, см., также теорему 8.2.1 на
с.290 в
Vinter \cite{vin00}.
 Можно также применить следствие
3.6.1 на с.142 в Кларк \cite{Cla83-r}, где функция $f$ -- это значение
дополнительной переменной, так что $\zeta$ -- это 
$(k+2)$--вектор, у которого последний элемент равен единице, а все
остальные -- нулю, откуда следует, что число $\lambda$  в условии 4)
на с.142 в Кларк \cite{Cla83-r} не влияет на  значения
первых $k+1$ компонент в левой части. Используется также утверждение
предложения 2.4.4 на с.55 в Кларк \cite{Cla83-r}, в соответствии с
которым $\partial d_{S'}(T^\ast,x^\ast(T^\ast))$ может быть заменено на
$N_{S'}(T^\ast,x^\ast(T^\ast))$\,.  В условиях леммы \ref{le:lagrange}
$S'=\R_+\times S$\,. Поэтому первая компонента всех векторов из
$N_{S'}(T^\ast, x^\ast(T^\ast))$ равна нулю, например, по теореме
 2.5.6 на с.67  в Кларк \cite{Cla83-r}.  Как следствие условия
 трансверсальности   4., $h=0$\,. То, что функция 
 $H$ из
 \eqref{eq:31} может использоваться  как функция $H$ следствия 3.6.1
 на с.142 в Кларк \cite{Cla83-r},
  устанавливается рассуждениями аналогичными использованным при  доказательстве теоремы 4.2.2 на с.156 в
 Кларк \cite{Cla83-r}, где условие строгой липшицевости $H$ заменено
 на условие локальной липшицевости $L$\,.
Результат аналогичный следствию 3.6.1
 на с. 142 в Кларк \cite{Cla83-r} содержится
 также в теореме 10.5.1 на с.383 в Vinter \cite{vin00}.
\subsection{Доказательство \eqref{eq:35}}
\label{sec:-hamil}

Заметим, что
$  \pi(u)=\sup_{\theta\in\R}( u\theta-(e^\theta-1))\,$, где $u\ge0$\,.
Как следствие, применяя теорему о минимаксе к \eqref{eq:35a},
\begin{multline*}
L_J(y)=  \inf_{\substack{(a,d,\varrho)\in\R_{+}^K
\times\R_+^K\times\mathbb{S}_+^{K\times K}
:\\y=a+(\varrho^T-I)d}}\psi_J(a,d,\varrho)
= \inf_{\substack{(a,d,\varrho)\in\R_{+}^K
\times\R_+^K\times\mathbb{S}_+^{K\times K}
:\\y=a+(\varrho^T-I)d}}\Bl(\sum_{k=1}^K
\sup_\theta\bl(\theta\,a_k-(e^\theta-1)\lambda_k\br)\,+
\\+\sum_{k\in J^c}
\sup_\theta\bl(\theta\,d_k-(e^\theta-1)\mu_k\br)\,
+\sum_{k\in
  J}\sup_\theta\bl(\theta\,d_k-(e^\theta-1) \mu_k\br)\,
\,
\ind_{(\mu_k,\infty)}(d_k)\Br)+
\\+\sum_{k=1}^K d_k \biggl[\sum_{l=1}^K 
\sup_\theta\bl(\theta\,\varrho_{kl}-(e^\theta-1)p_{kl}\br)\,
+
\sup_\theta\bl(\theta\,\bl(1-\sum_{l=1}^K
  \varrho_{kl}\br)-(e^\theta-1)
\bl(1-\sum_{l=1}^K
p_{kl}\br)\br)\,\biggr]\Br)=\\=
 \inf_{\substack{(a,d,\varrho)\in\R_{+}^K
\times\R_+^K\times\mathbb{S}_+^{K\times K}
:\\y=a+(\varrho^T-I)d}}\sup_{\substack{\theta_k,\,
\vartheta_k,\,\sigma_{kl},\,
\tau_k,\,\\k,l\in \{1,2,\ldots,K\}}}
\Bl(\sum_{k=1}^K
\bl(\theta_k\,a_k-(e^{\theta_k}-1)\lambda_k\br)\,
+\sum_{k\in J^c}
\bl(\vartheta_k\,d_k-(e^{\vartheta_k}-1)\mu_k\br)\,+
\\
+\sum_{k\in
  J}\bl(\vartheta_k\,d_k-(e^{\vartheta_k}-1) \mu_k\br)\,
\,
\ind_{(\mu_k,\infty)}(d_k)
+\sum_{k=1}^K d_k \biggl[\sum_{l=1}^K 
\bl(\sigma_{kl}\,\varrho_{kl}-(e^{\sigma_{kl}}-1)p_{kl}\br)\,+
\\+
\tau_k\,\bl(1-\sum_{l=1}^K
  \varrho_{kl}\br)-(e^{\tau_k}-1)
\bl(1-\sum_{l=1}^K
p_{kl}\br)\,\biggr]\Br)=\\=
\sup_{\substack{\theta_k,\,
\vartheta_k,\,
\sigma_{kl},\,
\tau_k,\,\\k,l\in\{1,2,\ldots,K\}}}
  \inf_{\substack{(a,d,\varrho)\in\R_{+}^K
\times\R_+^K\times\mathbb{S}_+^{K\times K}
:\\y=a+(\varrho^T-I)d}}
\Bl(\sum_{k=1}^K
\theta_k\,a_k
+\sum_{k\in J^c}
\vartheta_k\,d_k\,+
\\
+\sum_{k\in
  J}\vartheta_k\,d_k\,
\ind_{(\mu_k,\infty)}(d_k)
+\sum_{k=1}^K d_k \sum_{l=1}^K 
\sigma_{kl}\,\varrho_{kl}\,+
\\+\sum_{k=1}^K d_k
\tau_k\,\bl(1-\sum_{l=1}^K
  \varrho_{kl}\br)
-\sum_{k=1}^K(e^{\theta_k}-1)\lambda_k
-\sum_{k\in J^c}(e^{\vartheta_k}-1)\mu_k\,
-\sum_{k\in
  J}(e^{\vartheta_k}-1) \mu_k
\,
\ind\bl(d_k>\mu_k\br)-\\
-\sum_{k=1}^K d_k \sum_{l=1}^K (e^{\sigma_{kl}}-1)p_{kl}
-\sum_{k=1}^K d_k(e^{\tau_k}-1)
\bl(1-\sum_{l=1}^K
p_{kl}\br)
\Br)=\\=\sup_{\substack{\theta_k,\,
\vartheta_k,\,\sigma_{kl},\,
\tau_k,\,\\k,l\in \{1,2,\ldots,K\}}}
 \inf_{\substack{(d,\varrho)\in\R_+^K\times\mathbb{S}_+^{K\times K}
:\\y\ge(\varrho^T-I)d}}
\Bl(\sum_{k=1}^K
\theta_k\,y_k-
\sum_{k=1}^K\sum_{l=1}^K
\theta_l \varrho_{kl}d_k+\sum_{k=1}^K
\theta_kd_k
+\sum_{k\in J^c}
\vartheta_k\,d_k\,+
\\
+\sum_{k\in
  J}\vartheta_k\,d_k\,
\ind_{(\mu_k,\infty)}(d_k)
+\sum_{k=1}^K d_k \sum_{l=1}^K 
\sigma_{kl}\,\varrho_{kl}
+\sum_{k=1}^K d_k
\tau_k\,\bl(1-\sum_{l=1}^K
  \varrho_{kl}\br)-\\
-\sum_{k=1}^K(e^{\theta_k}-1)\lambda_k
-\sum_{k\in J^c}(e^{\vartheta_k}-1)\mu_k\,
-\sum_{k\in
  J}(e^{\vartheta_k}-1) \mu_k
\,
\ind_{(\mu_k,\infty)}(d_k)-\\
-\sum_{k=1}^K d_k \sum_{l=1}^K (e^{\sigma_{kl}}-1)p_{kl}
-\sum_{k=1}^K d_k(e^{\tau_k}-1)
\bl(1-\sum_{l=1}^K
p_{kl}\br)
\Br)\,.
\end{multline*}
Найдём инфимум по $\varrho$\,. 
Нужно минимизировать
$\sum_{k=1 }^K\sum_{l=1}^K\bl(-
\theta_l +
\sigma_{kl}-\tau_k
\br)\,d_k\varrho_{kl}$
при условии, что
 $\sum_{k=1}^K \varrho_{kl}d_k\le y_l+d_l$\,, $\varrho_{kl}\ge0$\,,
и $\sum_{l=1}^K\varrho_{kl}\le1$\,.
В соответствии с методом множителей Лагранжа, см., напр., теорему
6.2.4 на с.196 в Базара и Шетти \cite{BazShe82}, существует
неотрицательный вектор
$\alpha=(\alpha_1,\ldots,\alpha_K)$ такой, что достигается минимум
выражения 
\begin{equation*}
\sum_{k=1 }^K\sum_{l=1}^K\bl(-
\theta_l +
\sigma_{kl}-\tau_k
+ \alpha_l 
\br)\,d_k\varrho_{kl}
-\sum_{l=1}^K\alpha_l(y_l+d_l)
\end{equation*}
при условии, что
 $\varrho_{kl}\ge0$ и $\sum_{l=1}^K\varrho_{kl}\le1$\,. 
Если $\min_{l}\bl(-
\theta_l +
\sigma_{kl}-\tau_k
+ \alpha_l 
\br)\le 0$\,, то нужно взять $\varrho_{km}=1$\,, где 
$(-
\theta_m +
\sigma_{km}-\tau_k
+ \alpha_m)
=\min_l\bl(-
\theta_l +
\sigma_{kl}-\tau_k
+ \alpha_l 
\br)$ и $\varrho_{kl}=0$ при $l\not=m$\,.
Если $\min_{l}\bl(-
\theta_l +
\sigma_{kl}-\tau_k
+ \alpha_l 
\br)> 0$\,, то $\varrho_{kl}=0$\,.
Получаем, что минимум равен
$  \sum_{k }d_k\min_l\bl(-
\theta_l +
\sigma_{kl}-\tau_k
+ \alpha_l 
\br)\wedge 0-\sum_k\alpha_k(y_k+d_k)$\,,
где использовано обозначение $u\wedge0=\min(u,0)$\,.
Максимум по $\alpha$ достигается при $\alpha=0$\,. Это и есть нужное $\alpha$\,. 

Таким образом, обозначая
\begin{equation}\label{eq:2}
V_k=
\sum_{l=1}^K (e^{\sigma_{kl}}-1)p_{kl}
+(e^{\tau_k}-1)
\bl(1-\sum_{l=1}^K
p_{kl}\br)+\max_l\bl(
\theta_l -
\sigma_{kl}
\br)\vee (-{\tau_k})\,,
\end{equation}
где используется обозначение $u\vee v=\max(u,v)$\,,
получаем, что
\begin{multline}
  \label{eq:5}
L_J(y)
=\sup_{\substack{\theta_k,\,
\vartheta_k,\,\sigma_{kl},\,
\tau_k,\,\\k,l\in \{1,2,\ldots,K\}}}
 \inf_{d\in\R_+^K}
\Bl(\sum_{k=1}^K
\theta_k\,y_k-\sum_{k=1}^KV_kd_k
+\sum_{k\in
  J}\bl(\vartheta_k\ind_{(\mu_k,\infty)}(d_k)+\theta_k
\br)\,d_k\,+
\\+\sum_{k\in J^c}\bl(\theta_k+
\vartheta_k\br)\,d_k
-\sum_{k=1}^K(e^{\theta_k}-1)\lambda_k
-\sum_{k\in J^c}(e^{\vartheta_k}-1)\mu_k\,
-\sum_{k\in
  J}(e^{\vartheta_k}-1) \mu_k
\,
\ind_{(\mu_k,\infty)}(d_k)\Br)=\\
=\sup_{\substack{\theta_k,\,
\vartheta_k,\,\sigma_{kl},\,
\tau_k\,:\\\theta_k+\vartheta_k\ge V_k}}
 \inf_{d\in\R_+^K}
\Bl(\sum_{k=1}^K
\theta_k\,y_k-\sum_{k=1}^KV_kd_k
+\sum_{k\in
  J}\bl(\vartheta_k\ind_{(\mu_k,\infty)}(d_k)+\theta_k
\br)\,d_k\,+
\\+\sum_{k\in J^c}\bl(\theta_k+
\vartheta_k\br)\,d_k
-\sum_{k=1}^K(e^{\theta_k}-1)\lambda_k
-\sum_{k\in J^c}(e^{\vartheta_k}-1)\mu_k\,
-\sum_{k\in
  J}(e^{\vartheta_k}-1) \mu_k
\,
\ind_{(\mu_k,\infty)}(d_k)\Br)=\\
= \sup_{\substack{\theta_k,\,\vartheta_k,\,
\sigma_{kl},\,\tau_k:\\\theta_k+\vartheta_k\ge V_k}}
\Bl(\sum_{k=1}^K
\theta_k\,y_k+\sum_{k\in J }\mu_k
\bl(\theta_k+\vartheta_k-V_k-(e^{\vartheta_k}-1)\br)
\wedge0
-\sum_{k=1}^K(e^{\theta_k}-1)\lambda_k-\sum_{k\in J^c}(e^{\vartheta_k}-1)\mu_k\,
\Br)\,.
\end{multline}
Поясним, как получается последнее равенство. 
Если брать инфимум членов с $d_k$ по 
$d_k>\mu_k$ для $k\in J$\,, то получится 
$\mu_k\bl(-V_k+\theta_k+\vartheta_k
-(e^{\vartheta_k}-1)\br)$\,.
Если брать инфимум по $d_k\le\mu_k$\,, то получится 
$\mu_k(\theta_k-V_k)\wedge0$ \,.
Поскольку $\vartheta_k -(e^{\vartheta_k}-1)\le0$\,, то минимум
этих выражений равен
$\mu_k\bl(-V_k+\theta_k+\vartheta_k
-(e^{\vartheta_k}-1)\br)\wedge0$\,.

Проминимизируем $V_k$ по $\sigma_{kl}\,,{\tau_k}$\,. В силу \eqref{eq:2},
эта функция  выпукла по $(\sigma_{kl},\tau_k)$\,.
Если $p_{kl}=0$ для некоторого $l$\,, то можно положить
$\sigma_{kl}=\infty$\,, что позволяет ограничиться рассмотрением тех
$l$\,, для которых $p_{kl}>0$\,. Аналогично, можно исключить из
рассмотрения $\tau_k$\,,  если 
$\sum_{l=1}^K
p_{kl}=1$\,, полагая $\tau_k=\infty$\,. Поэтому минимум  правой части
\eqref{eq:2} существует. В этой точке
субдифференциал содержит нулевой вектор. Обозначим через ${\beta_k}$ значение
максимума в правой части \eqref{eq:2}.  Субдифферeнциал правой части
\eqref{eq:2}  по
$((\sigma_{kl},\,l\in\{1,2,\ldots,K\})\,,\tau_k)$ равен
$((e^{\sigma_{kl}}p_{kl}\,,
l\in\{1,2,\ldots,K\}),e^{\tau_k}(1-\sum_{l=1}^Kp_{kl}))
-\text{co}\,((\ind_U(l),\,\{l\in1,2,\ldots,K\}),\ind_{\beta_k}(-\tau_k))$\,, 
где $\text{co}$ обозначает операцию взятия выпуклой оболочки и $U$ -- множество
тех $l$ (возможно, пустое), для которых
$\theta_l-\sigma_{kl}={\beta_k}$\,.
Так как в точке минимума субдифференциал содержит нулевой вектор, то
 $U$ содержит все $l$ такие, что $p_{kl}>0$\,, и
 $\tau_k=-{\beta_k}$\,, если $\sum_{l=1}^Kp_{kl}<1$\,.
Поэтому существуют неотрицательные
$\alpha_1,\,\ldots,\alpha_K,\alpha_{K+1}$\,, сумма которых равна 1\,,
такие что
$e^{\sigma_{kl}}p_{kl}-\alpha_l=0$ для $l\in U$\,, $\alpha_l=0$ для $l\not\in U$\,,
$e^{\tau_k}(1-\sum_{l=1}^Kp_{kl})-\alpha_{K+1}=0$ при условии, что
$\tau_k=-{\beta_k}$\,, и  $\alpha_{K+1}=0$ 
в противном случае.
Таким образом,  $\alpha_l=e^{\sigma_{kl}}p_{kl}$ 
 для всех $l\in\{1,2,\ldots,K\}$ и $\alpha_{K+1}=e^{\tau_k}(1-\sum_{l=1}^Kp_{kl})$\,.
Получаем, что
$\sum_{l=1}^Ke^{\sigma_{kl}}p_{kl}+e^{\tau_k}(1-\sum_{l=1}^Kp_{kl})=1$\,. 
Так как $\theta_l-\sigma_{kl}={\beta_k}$\,, если $p_{kl}>0$\,,
и $\tau_k=-{\beta_k}$\,, если $1-\sum_{l=1}^Kp_{kl}>0$\,, то
$\sum_{l=1}^Ke^{-{\beta_k}+\theta_l}p_{kl}+e^{-{\beta_k}}(1-\sum_{l=1}^Kp_{kl})=1$\,, т.е.,
$e^{\beta_k}=\sum_{l=1}^Ke^{\theta_l}p_{kl}+1-\sum_{l=1}^Kp_{kl}$\,. Окончательно
получаем, что
${\beta_k}=\ln(\sum_{l=1}^Ke^{\theta_l}p_{kl}+1-\sum_{l=1}^Kp_{kl})$ 
и  минимум $V_k$ равен
$\sum_{l=1}^K(e^{-{\beta_k}+\theta_l}-1)p_{kl}
+(e^{-{\beta_k}}-1)(1-\sum_{l=1}^Kp_{kl})+{\beta_k}={\beta_k}$\,.

Таким образом, самое последнее выражение  в  \eqref{eq:5} равно
  \begin{multline}
\label{eq:6} \sup_{\theta_k,\,\vartheta_k:\,
\theta_k+
\vartheta_k\ge\beta_k}
\Bl(\sum_{k=1}^K
\theta_k\,y_k
+\sum_{k\in J}\mu_k\bl(\theta_k
+\vartheta_k-\beta_k-(e^{\vartheta_k}-1)\br)\wedge0 -\\
  -\sum_{k=1}^K(e^{\theta_k}-1)\lambda_k-\sum_{k\in J^c}(e^{\vartheta_k}-1)\mu_k\,
\Br)\,.
\end{multline}
Если $\theta_k-\beta_k\ge0$\,, где $k\in J$\,,
 то  выбор $\vartheta_k=0$ показывает, что  $\sup_{\vartheta_k}$
 множителя при $\mu_k$ равен нулю.
Поэтому \eqref{eq:6}
 переписывается в виде
\begin{multline*}
\sup_{\theta_k,\,\vartheta_k:\,
\theta_k+
\vartheta_k\ge\beta_k}
\Bl(\sum_{k=1}^K
\theta_k\,y_k
+\sum_{k\in J}\ind_{(-\infty,\beta_k)}(\theta_k)\mu_k\bl(\theta_k
+\vartheta_k-\beta_k-(e^{\vartheta_k}-1)\br)\wedge0 -\\\,
-\sum_{k=1}^K(e^{\theta_k}-1)\lambda_k-\sum_{k\in J^c}(e^{\vartheta_k}-1)\mu_k\,
\Br)\,.
\end{multline*}
Так как функция $u-(e^u-1)$ убывает при $u\ge0$\,, то супремум по
$\vartheta_k$ для $k\in J$ достигается на границе области
ограничений. Соотношение \eqref{eq:35} следует теперь из
определения $h_k(\theta)$ в
\eqref{eq:27}.

\subsection{Доказательство леммы \ref{le:sush_os'}}
\label{sec:-lemm}
  Так как $C=(I-P^T)^{-1}$\,, то
\begin{equation*}
  c_{ml}=\frac{1}{\text{det}\,(I-P^T)}\,(-1)^{m+l}M_{lm}\,.
\end{equation*}
где $M_{lm}$ -- это $(l,m)$--минор матрицы $I-P^T$\,. Заметим, что
$\text{det}\,(I-P^T)>0$\,. Действительно, $\text{det}\,I=1$ и 
$\text{det}\,(I-\lambda P^T)\not=0$ при $\lambda\in[0,1]$\,, поскольку
спектральный радиус $P$ меньше единицы. По непрерывности, 
$\text{det}\,(I-P^T)>0$\,. Таким образом, требуется доказать, что 
$(-1)^{m+l}M_{lm}\le M_{mm}$\,.
Предположим, что $l=m+1$\,. Нужно доказать, что
$M_{mm}+M_{m+1,m}\ge0$\,.
Ввиду мультилинейности детерминанта, $M_{mm}+M_{m+1,m}$ -- это
детерминант матрицы  $\tilde I-\tilde P$\,, где $\tilde I$ --
единичная $(K-1)\times (K-1)$--матрица, а $\tilde P$ -- $(K-1)\times
(K-1)$--матрица, которая получается из $P^T$ сложением строк $m$ и
$m+1$ с последующим вычеркиванием $m$--го столбца. 
Так как матрица $P^T$ является субстохастической по столбцам, то
матрица $\tilde P$ также является субстохастической по столбцам.
Следовательно, её спектральный радиус не превосходит единицы. Отсюда
следует, что 
 $\text{det}\,(\tilde I-\tilde P)\ge0$\,, т.е.,
$M_{m+1,m}+ M_{m,m}\ge0$\,. Пусть $l>m+1$\,. С помощью
последовательности  транспозиций соседних строк и столбцов
передвинем  в матрице $\tilde I-\tilde P$ строку
 $l$  и столбец $l$ в положение строки $m+1$ и столбца $m+1$\,, 
соответственно, сохраняя при этом
 взаимное расположение остальных строк и столбцов.
Эта матрица, по-прежнему, имеет
 вид $\tilde I-\tilde{\tilde P}$\,.
В ней $\tilde M_{m+1,m}=(-1)^{l-m-1}M_{lm}$\,, т.к. знак минора
меняется только при транспозиции столбцов, и 
$\tilde M_{mm}=(-1)^{2(l-m-1)}M_{mm}$\,. Так как $\tilde M_{mm}\ge
(-1)\tilde M_{m+1,m}$\,, то $M_{mm}\ge(-1)^{l-m-2} M_{lm}=(-1)^{l+m}
M_{lm}$\,.
Случай $l<m$ рассматривается аналогично.
\bibliographystyle{unsrt}
\def\cprime{$'$} \def\cprime{$'$} \def\cprime{$'$}
  \def\polhk#1{\setbox0=\hbox{#1}{\ooalign{\hidewidth
  \lower1.5ex\hbox{`}\hidewidth\crcr\unhbox0}}} \def\cprime{$'$}
  \def\cprime{$'$} \def\cprime{$'$} \def\cprime{$'$} \def\cprime{$'$}
  \def\cprime{$'$} \def\cprime{$'$} \def\cprime{$'$}

\end{document}